\documentclass[journal]{IEEEtran}
\usepackage{graphicx}
\usepackage{subfiles}
\IEEEoverridecommandlockouts
\ifCLASSINFOpdf
\usepackage{mathrsfs}
\usepackage{color}
\usepackage{amsfonts,amssymb}
\usepackage{subfigure}
\usepackage{booktabs}
\usepackage{caption}
\usepackage{multirow}
\usepackage{tabularx}
\usepackage{bbding}

\usepackage{float} 
\usepackage{amsmath}
\usepackage{enumerate}
\usepackage{algorithm}
\usepackage{algorithmic}
\usepackage{setspace}
\usepackage{bm}
\usepackage{makecell}
\usepackage{epstopdf}
\usepackage{mathtools}
\usepackage{bbm}
\usepackage{dsfont}
\usepackage{pifont}
\usepackage{amsthm}
\usepackage{cite}
\usepackage{balance}
\usepackage[T1]{fontenc}
\usepackage{extarrows}

\graphicspath{{pic}}

\newtheorem{remark}{Remark}
\newtheorem{example}{Example}
\newtheorem{theorem}{Theorem}
\newtheorem{lemma}{Lemma}
\newtheorem{corollary}{Corollary}
\newtheorem{assumption}{Assumption}
\newtheorem{problem}{Problem}
\newtheorem{definition}{Definition}

\newcommand\R{\mathbb{R}}
\newcommand\pr{\operatorname{Pr}}

\newcommand\E{\mathbb{E}}
\newcommand\cov{\operatorname{Cov}}

\newcommand\dd{\mathrm{d}}
\newcommand\tr{\operatorname{tr}}

\begin{document}
\title{\LARGE \bf Robust Probabilistic Prediction for Stochastic Dynamical Systems}
\author{Tao Xu and Jianping He
   \thanks{The authors are with the Department of Automation, Shanghai Jiao Tong University, and Key Laboratory of System Control and Information Processing, Ministry of Education of China, Shanghai 200240, China. E-mail: \{Zerken, jphe\}@sjtu.edu.cn.}}

\maketitle

\begin{abstract}
   It is critical and challenging to design robust predictors for stochastic dynamical systems (SDSs) with uncertainty quantification (UQ) in the prediction. Specifically, robustness guarantees the worst-case performance when the predictor's information set of the system is inadequate, and UQ characterizes how confident the predictor is about the predictions. However, it is difficult for traditional robust predictors to provide robust UQ because they were designed to robustify the performance of point predictions.
   In this paper, we investigate how to robustify the probabilistic prediction for SDS, which can inherently provide robust distributional UQ. To characterize the performance of probabilistic predictors, we generalize the concept of likelihood function to likelihood functional, and prove that this metric is a proper scoring rule. Based on this metric, we propose a framework to quantify when the predictor is robust and analyze how the information set affects the robustness. Our framework makes it possible to design robust probabilistic predictors by solving functional optimization problems concerning different information sets. In particular, we design a class of moment-based optimal robust probabilistic predictors and provide a practical Kalman-filter-based algorithm for implementation. Extensive numerical simulations are provided to elaborate on our results.
\end{abstract}


\section{Introduction}
\subsection{Background}
Stochastic dynamical systems (SDSs) play a critical role in deepening our comprehension of the changing world full of uncertainties. Within the analysis of SDS, there is a significant need to predict the system outputs, which is crucial across various fields, including climate science, robotics, and finance. When the predictor's information set of the system is inadequate, designing robust predictors helps to guarantee the worst-case prediction performance.

In addition to robustness, it attracts increasing attention to provide uncertainty quantification (UQ) for the prediction. Because the prediction serves as a fundamental basis for many subsequent algorithms, a UQ of high quality can provide more side information to improve their performances. For example, a popular line of recent research incorporates predictions in the design of algorithms such as online learning \cite{khodakLearningPredictionsAlgorithms2022}, smart optimization \cite{elmachtoubSmartPredictThen2022}, and online optimal control \cite{liOnlineOptimalControl2019}.

To ensure a predictor performs well regardless of its inadequate information set of an SDS, tremendous efforts have been made to robustify the Kalman filter \cite{bertsekasRecursiveStateEstimation1971,xieRobustKalmanFiltering1994,hassibiLinearEstimationKrein1996,shenGameTheoryApproach1997,karlgaardHuberBasedDividedDifference2007,gandhiRobustKalmanFilter2010,wangRobustInformationFilter2016,chenMaximumCorrentropyKalman2017a, huangNovelOutlierRobustKalman2021,shafieezadehabadehWassersteinDistributionallyRobust2018,wangRobustStateEstimation2021,wangDistributionallyRobustState2022a}. However, it is difficult for traditional robust predictors to simultaneously provide robust UQ because they were originally designed to robustify the performance of point predictions. Moreover, the existing UQs for the SDS predictor usually need relatively strong information about the system dynamics \cite{spallAsymptoticDistributionTheory1984,spallKantorovichInequalityError1995,maryakUncertaintiesRecursiveEstimators1995a,maryakUseKalmanFilter2004a,branickiQuantifyingBayesianFilter2014}, thus incurring a trade-off between robustness and the quality of UQ. To ensure the robustness of prediction and UQ simultaneously, we ask:
\textit{How to design robust predictors for SDS with UQ?}

\subsection{Motivations}
To design robust predictors for SDS with UQ, a direct idea is ``first-robust-then-UQ'' based on previous works. That is, a robust predictor is first used to make a point prediction, and then UQ is provided. However, traditional robust predictors were originally designed for point prediction, and most of them hold relatively stringent assumptions on the posterior distributions (e.g., the first two moments are finite). Therefore, the quality and robustness of UQ cannot be guaranteed.

Our method is to think in reverse: ``first-UQ-then-robust''. Specifically, we consider probabilistic prediction for SDS and then robustify it. Probabilistic prediction is of significant importance in various fields. For example, in transportation planning and management, a probabilistic predictor can provide probability distribution of traffic conditions at different times and locations, thus capturing different modes of behavior \cite{knittelDiPAProbabilisticMultiModal2023,huangMultimodalTrajectoryPrediction,fengReviewComparativeStudy2022}. This idea works because a probabilistic predictor can inherently provide UQ by predicting distributions rather than a single point \cite{robertsProbabilisticPrediction1965b,gneitingEditorialProbabilisticForecasting2008,buizzaValueProbabilisticPrediction2008,gneitingProbabilisticForecasting2014}, and the robustness of UQ can also be guaranteed. Then, our problem can be further specified as:
\textit{How to design robust probabilistic predictors for SDS?}

\subsection{Challenges}
To design robust probabilistic predictors for SDS, we face some new challenges.

First, before trying to guarantee the worst-case performance, a metric that measures the performance of a probabilistic predictor needs to be specified. Theoretically, this metric should be a proper scoring rule \cite{gneitingProbabilisticForecasting2014} that assesses calibration and sharpness simultaneously. Practically, this metric should be a local scoring rule \cite{gneitingStrictlyProperScoring2007a} (i.e., depends only on the predictive distribution and the realized observation), thus can be calculated without the need to know the ground-truth distribution. Additionally, since predicting the trajectory of SDS is an online algorithm, the metric should also support an easy online implementation.

Second, the meaning of robustness should be specified in the context of probabilistic prediction. The performance of a probabilistic predictor is deeply affected by the information set. If the information about the system is too strong, it may be too optimistic about some trajectories. Consequently, a robust predictor designed based on this information is no longer robust. For example, inappropriately assuming a heavy-tailed noise to be Gaussian will lead to significant performance degradation under the classical settings of a Kalman filter. If the information set is too weak, poor performance happens because the predictor takes those trajectories with very small possibilities into consideration.

Third, a real-world probabilistic predictor's information set of an SDS can be very restrictive, subjective and time-variant. For example, it may only know the value of some lower-order moments or the support of these distributions rather than the probability density function. Even worse, the predictor may just have a subjective belief that the moments under a certain order exist rather than knowing their exact values. Since both the knowledge and belief of $\Phi$ can be updated as more outputs are generated and observed, the information set is time-variant.

\subsection{Contributions}
The contributions of this paper are as follows:
\begin{itemize}
   \item We propose a metric that is both theoretically proper and practically implementable to measure the performance of a probabilistic predictor for SDS. Theoretically, it generalizes the concept of likelihood function to likelihood functional, and is proved to be a proper scoring rule. Practically, it measures the log-likelihood that a trajectory can be predicted by a probabilistic predictor, and can be easily updated online.
   \item We propose a functional-optimization-based framework to quantify when a probabilistic predictor is robust. Based on this framework, a paradigm for designing robust probabilistic predictors is provided. Then we analyze how the restrictiveness and subjectiveness of a predictor's information set of an SDS affect the robustness.
   \item We design a class of moment-based robust probabilistic predictors when the information set is restricted to the moment knowledge. Moreover, we derive their optimal form concerning different information sets. Finally, we implement a moment-based robust online probabilistic predictor based on the Kalman filter, which can adaptively adjust its information set.
\end{itemize}

The remainder of this paper is organized as follows. Section II introduces the related works. Section III introduces some preliminaries on moment and entropy, then formulates the problem of interest.  Sec. IV defines the log-likelihood functional and verifies the optimality condition. Sec. V proposes a framework to define what is a robust probabilistic predictor, derives the necessary conditions for a class of moment-based robust probabilistic predictors and solves their optimal forms. Based on this framework, Sec. VI implements a complete moment-based robust online probabilistic predictor by integrating the Kalman filter. Sec. VII shows simulation results and analysis. Sec. VIII presents concluding remarks.

\section{Related Works}
Within the prediction research of SDS, there has been extensive research on designing robust predictors and facilitating predictors with UQ. This section gives a brief overview.
\paragraph{Robust predictor}
When the predictor's information set about the system is inadequate, robustness is needed to guarantee the worst-case performance. Therefore, a prior assumption that each robust predictor should declare is the content of its information set. Classified by the types of information set, there are parameter-robust predictors \cite{bertsekasRecursiveStateEstimation1971,xieRobustKalmanFiltering1994,hassibiLinearEstimationKrein1996,shenGameTheoryApproach1997}, outlier-robust predictors \cite{karlgaardHuberBasedDividedDifference2007,gandhiRobustKalmanFilter2010,wangRobustInformationFilter2016,chenMaximumCorrentropyKalman2017a, huangNovelOutlierRobustKalman2021}, distributionally-robust predictors \cite{shafieezadehabadehWassersteinDistributionallyRobust2018,wangRobustStateEstimation2021,wangDistributionallyRobustState2022a}, etc. Since robust predictors were originally designed to guarantee the performance of point prediction rather than the performance of probabilistic prediction, the information sets under consideration are relatively strong in the context of probabilistic distribution. For example, nearly all of the robust predictors assume the existence of expectation, which is utilized as the predicted output. However, many heavy-tailed distributions cannot guarantee the existence of expectation, e.g., the Cauchy distribution.

\paragraph{Prediction with UQ}
One of the most frequently used UQs for an SDS predictor is the covariance of the prediction error, which usually requires very strong assumptions on the system, e.g., linear dynamics and Gaussian noises. If the system is nonlinear and non-Gaussian, the covariances of prediction errors usually do not have explicit expressions \cite{maryakUseKalmanFilter2004a}. The uncertainty in prediction error can be quantified from other different perspectives, e.g., a scalar ratio error setting \cite{heffesEffectErroneousModels1966}, convergence and divergence analysis \cite{sangsuk-iamAnalysisDiscretetimeKalman1990}, ordering and relative closeness for three mean square error (MSE) based metrics \cite{gePerformanceAnalysisKalman2016}, to name a few.
Apart from exactly characterizing the prediction error, many works contributed by approximating the prediction error with probabilistic inequalities \cite{maryakUseKalmanFilter2004a,spallKantorovichInequalityError1995,weiUncertaintyQuantificationExtended2022}. Another group of work utilizes the asymptotic Gaussian assumption for further asymptotic characterization \cite{maryakUncertaintiesRecursiveEstimators1995a,spallAsymptoticDistributionTheory1984,alievEvaluationConvergenceRate1999a}.
All these point predictors with probabilistic UQs have motivated the ideas of probabilistic interval prediction and probabilistic prediction, which are the most popular techniques for uncertainty quantification \cite{abdarReviewUncertaintyQuantification2021}. A probabilistic interval predictor predicts the outcome by an interval with high probability, see \cite{bravoCombinedStochasticDeterministic2015,carnereroProbabilisticIntervalPredictor2022,mirasierraPredictionErrorQuantification2022} and references therein.
To provide further information for the prediction, a natural extension of the probabilistic interval predictor is the probabilistic predictor. Because the posterior distributions of SDS are typically intractable to have explicit expression, it is challenging to provide probabilistic distributional UQ for SDS. Approximated Bayesian inference methods such as variational Bayesian inference \cite{chappellVariationalBayesianInference2009,daunizeauVariationalBayesianIdentification2009,smidlVariationalBayesianFiltering2008,liRobustVariationalBasedKalman2021} and sequential Monte Carlo methods \cite{djuricParticleFiltering2003,doucetTutorialParticleFiltering2009,arulampalamTutorialParticleFilters2002,gustafssonParticleFilterTheory2010} can approximate the posterior state distributions. However, since these methods were originally developed for point prediction, the approximated distributions provided by them cannot guarantee the quality of UQ. Another line of probabilistic predictors is motivated by the safety certification \cite{wabersichProbabilisticModelPredictive2022a} requirements of stochastic model predictive control (SMPC), and \cite{landgrafProbabilisticPredictionMethods2023} has provided a thorough review for these predictors. Nevertheless, these probabilistic predictors can only be applied to the SDSs that are perfectly observed.

\begin{table}[t]
   \centering
   \caption{Comparison of Works on Predictors for SDS}
   \label{tb:related-works}
   \resizebox{\linewidth}{!}{
      \begin{tabular}{cccccc}
         \toprule
         \textbf{Works}                                                     & \cite{gePerformanceAnalysisKalman2016}               & \cite{maryakUseKalmanFilter2004a}                          & \cite{carnereroProbabilisticIntervalPredictor2022}         & \cite{landgrafProbabilisticPredictionMethods2023}              & this work                                                      \\
         \midrule
         \textbf{\begin{tabular}{c} Partially \\ Observable \end{tabular}}  & \Checkmark                                           & \Checkmark                                                 & \Checkmark                                                 & \XSolidBrush                                                   & \Checkmark                                                     \\
         \midrule
         \textbf{\begin{tabular}{c} Type of \\ Predictor \end{tabular}}     & \begin{tabular}{c} point \\ prediction \end{tabular} & \begin{tabular}{c} point \\ prediction \end{tabular}       & \begin{tabular}{c} probabilistic \\ interval \end{tabular} & \begin{tabular}{c} probabilistic \\ distribution \end{tabular} & \begin{tabular}{c} probabilistic \\ distribution \end{tabular} \\
         \midrule
         \textbf{\begin{tabular}{c} Prediction \\ Robustness \end{tabular}} & \XSolidBrush                                         & \Checkmark                                                 & \Checkmark                                                 & \Checkmark                                                     & \Checkmark                                                     \\
         \midrule
         \textbf{\begin{tabular}{c} Type of \\ UQ \end{tabular}}            & MSE                                                  & \begin{tabular}{c} probabilistic \\ interval \end{tabular} & \begin{tabular}{c} probabilistic \\ interval \end{tabular} & \begin{tabular}{c} probabilistic \\ distribution \end{tabular} & \begin{tabular}{c} probabilistic \\ distribution \end{tabular} \\
         \midrule
         \textbf{\begin{tabular}{c} UQ \\ Robustness \end{tabular}}         & \Checkmark                                           & \XSolidBrush                                               & \XSolidBrush                                               & \XSolidBrush                                                   & \Checkmark                                                     \\
         \bottomrule
      \end{tabular}}
\end{table}

\section{Preliminaries and Problem Formulation}
\subsection{Preliminaries and Notations}
\subsubsection{Random Vector and Moment}
In this paper, we use bold letters to distinguish random vectors from constant vectors. Let $\mathbf{x}\in \R^d$ be a random vector with probability density function (pdf) $p_\mathbf{x}(\cdot)$. Given $\alpha \in \R_+^d$, the $\alpha$-moment of $\mathbf{x}$ is defined as \[\mu_\alpha(\mathbf{x}) := \int s^\alpha p_\mathbf{x}(s) \dd s,\] where
$s^\alpha = \prod_{i=1}^{d} (s^{(i)})^{\alpha_i}$, and the superscript $(i)$ denotes the $i$th element of $s$.
The order of an $\alpha$-moment, denoted as $|\alpha|$, is the sum of all $\alpha^{(i)}$ such that $|\alpha| = \sum_{i=1}^{d}\alpha^{(i)}$. For example, the expectation of $\mathbf{x}$ is expressed as a vector containing all the first-order moments,
\begin{equation*}
   \begin{aligned}
      \E (\mathbf{x}) = & \begin{bmatrix}
                             \mu_{e_1}(\mathbf{x}) & \mu_{e_2}(\mathbf{x}) & \cdots & \mu_{e_d}(\mathbf{x})
                          \end{bmatrix}^\top,
   \end{aligned}
\end{equation*}
and the covariance of $\mathbf{x}$ is a matrix containing all the second-order moments,
\begin{equation*}
   \begin{aligned}
      \cov (\mathbf{x}) = & \begin{bmatrix*}
                               \mu_{e_1 + e_1}(\mathbf{x})&\mu_{e_1 + e_2}(\mathbf{x})&\cdots & \mu_{e_1 + e_d}(\mathbf{x})\\
                               \vdots &\vdots&&\vdots \\
                               \mu_{e_d + e_1}(\mathbf{x})&\mu_{e_d + e_2}(\mathbf{x})&\cdots & \mu_{e_d + e_d}(\mathbf{x})
                            \end{bmatrix*},
   \end{aligned}
\end{equation*}
where $e_i\in\R^d$ is a unit vector with the $i$-the element equals $1$.
For the convenience and unity of notation, we denote $\E \mathbf{x}$ and $\cov \mathbf{x}$ as $\mu_1(\mathbf{x})$ and $\mu_2(\mathbf{x})$ respectively.

Some distributions can be uniquely determined by a finite order of moments, e.g., Gaussian distribution can be uniquely determined by the first two moments.
The more orders of moments are known, the more accurately the distribution of $\mathbf{x}$ can be characterized. However, it is not always possible to describe a distribution by moments, e.g., when the order of a Student's t distribution is smaller than $2$, no covariance exists; when the order is smaller than $1$, even no expectation exists.


\subsubsection{Entropy}
The differential entropy of a random variable $\mathbf{x}$ with support $\mathcal{X}$ and pdf $p_\mathbf{x}(x)$ is,
$$
   \mathrm{H}(\mathbf{x}):=-\int_{x \in \mathcal{X}} p_\mathbf{x}(x) \log p_\mathbf{x}(x) \dd x .
$$
We denote a sequence as $(\cdot)_{1:k}:= (\cdot)_1, (\cdot)_2, \ldots, (\cdot)_k$.
Let $\mathbf{x}_{1:2}$ be a pair of random variables with the joint pdf $p_{\mathbf{x}_{1:2}}\left(x_{1:2}\right)$ and the support $\mathcal{X} \times \mathcal{X}$. The joint entropy of $\mathbf{x}_{1:2}$ is
$$
   \mathrm{H}\left(\mathbf{x}_{1:2}\right):=\int_{x_1 \in \mathcal{X}}\! \int_{x_2 \in \mathcal{X}} p_{\mathbf{x}_{1:2}}\!\left(x_{1:2}\right) \log p_{\mathbf{x}_{1:2}}\!\left(x_{1:2}\right) \dd x_1 \dd x_2.
$$
The KL-divergence measures how much distant $\mathbf{x}_2$ diverges away from $\mathbf{x}_1$, i.e.,
\[D_{KL}(\mathbf{x}_1 \| \mathbf{x}_2):=\int_{x\in\mathcal{X}} p_{\mathbf{x}_1}(x) \log \left(\frac{p_{\mathbf{x}_1}(x)}{p_{\mathbf{x}_2}(x)}\right) \dd x.\]

\subsubsection{Probabilistic Prediction and Proper Scoring Rules}
A probabilistic prediction is to predict a random vector $\textbf{y}$ with pdf $p_\mathbf{y} \in \mathcal{H}$ by a pdf $\hat{p}_{\mathbf{y}} \in \mathcal{H}$, where $\mathcal{H}$ is the prediction space.
A scoring rule assigns a numerical score $\mathrm{S}(\hat{p}_{\mathbf{y}}, y)$ to each pair $(\hat{p}_{\mathbf{y}}, y)$, where $y$ is a realized outcome of $\textbf{y}$. It is a local scoring rule if it depends on the predictive distribution only through its value at the event y that realizes.
We write the expected value of a scoring rule as
\begin{equation*}
   \mathrm{S}(\hat{p}_{\mathbf{y}}, p_{\mathbf{y}}):=\E_{y} \mathrm{S}(\hat{p}_{\mathbf{y}}, y).
\end{equation*}
A scoring rule $\mathrm{S}$ is proper under the prediction space $\mathcal{H}$ if
\begin{equation}\label{eq:proper_scoring_rule}
   \mathrm{S}(\hat{p}_{\mathbf{y}}, p_{\mathbf{y}}) \geq \mathrm{S}(p_{\mathbf{y}}, p_{\mathbf{y}})
\end{equation}
holds for all $\hat{p}_{\mathbf{y}}, p_{\mathbf{y}} \in \mathcal{H}$. It is strictly proper if and only if equation \eqref{eq:proper_scoring_rule} holds when $\hat{p}_{\mathbf{y}}= p_{\mathbf{y}}$.

\subsection{System Dynamic}
Consider a class of discrete-time nonlinear stochastic dynamical systems,
\begin{equation}\label{eq:sds}
   \Phi: \left\{\begin{aligned}
      \mathbf{x}_{k} & = f(\mathbf{x}_{k-1}, \mathbf{u}_{k-1}) + \mathbf{w}_{k-1} \\
      \mathbf{y}_k   & = h(\mathbf{x}_k) + \mathbf{v}_k,
   \end{aligned}\right.
\end{equation}
where $\mathbf{x}_k\in\R^{d_x}$ is the system state vector, $\mathbf{u}_k\in\R^{d_u}$ is the control input, and $\mathbf{w}_k\in\R^{d_x}$ is the independent process noises.
$\mathbf{y}_k\in\R^{d_y}$ is the observation vector of $x_k$, $\mathbf{v}_k\in\R^{d_y}$ is the independent observation noise, and there is no cross-correlation among $\mathbf{w}_k, \mathbf{v}_k$ and $\mathbf{u}_k$. The initial state $\mathbf{x}_0$ is also independent with $\mathbf{w}_k, \mathbf{v}_k$ and $\mathbf{u}_k$.

\begin{assumption}
   The system dynamics $f: \R^{d_x} \to \R^{d_x}, h: \R^{d_x}\to d_y$, and the distributions of $\mathbf{w}_{1:n}, \mathbf{v}_{1:n}, \mathbf{u}_{1:n}, \mathbf{x}_0$ may be unknown to the predictor.
\end{assumption}

\subsection{Problem in Interests}
Suppose a probabilistic predictor keeps observing the trajectory generated from an SDS.
At time step $k$, the trajectory $y_{1:k-1}$ is observed, and the next output $y_k$ will be generated from a conditional distribution $p_{\mathbf{y}_k \mid \mathbf{y}_{1:k-1}}(\cdot \mid y_{1:k-1})$. A one-step probabilistic predictor $\mathscr{F}_k$ possesses a information set $\mathcal{I}$ of the system,
\begin{equation*}
   \mathcal{I} := \left\{\text{estimations of }f, h; \text{ features of }\mathbf{w}_{1:n}, \mathbf{v}_{1:n}, \mathbf{u}_{1:n}, \mathbf{x}_0 \right\},
\end{equation*}
where the estimations of $f, h$ may differ from the ground truth, and the features of those random vectors may include their moments, quantiles, supports, etc.
It aims to predict the conditional distribution based on previous observations $y_{1:k-1}$ and the information set $\mathcal{I}$:
\begin{equation}
   \mathscr{F}_k(y_{1:k-1}; \mathcal{I}) = \hat{p}_{\mathbf{y}_k \mid \mathbf{y}_{1:k-1}}(\cdot \mid y_{1:k-1}).
\end{equation}
After the system $\Phi$ generates $y_k$ from the conditional distribution, the prediction performance is evaluated, the observed trajectory and the information set are updated and the next round of prediction continues. This recursive procedure is visualized in Fig. \ref{fig:problem}.

The prior tool for analyzing an online probabilistic prediction is a proper performance metric.
\begin{problem}
Measure the prediction performance of an online probabilistic predictor by a metric that i) is a proper local scoring rule and ii) can be easily implemented online.
\end{problem}

When $\mathcal{I}$ is adequate (i.e., the pdfs of $\mathbf{w}_{1:n}, \mathbf{v}_{1:n}, \mathbf{u}_{1:n}, \mathbf{x}_0$ are known and the estimations of $f,h$ are precise), the conditional distributions can be uniquely determined. However, the information set in practice can be:
\begin{itemize}
   \item restrictive: knowing the value of some lower-order moments or the support of these distributions rather than the probability density function.
   \item subjective: a subjective belief that the moments with a certain order exist rather than knowing their exact values.
   \item time-variant: the knowledge of $\Phi$ can be updated as more outputs are observed and analyzed.
\end{itemize}
To guarantee the worst-case prediction performance against the above-listed problems, a robust predictor is needed.
\begin{problem}
Quantitatively specify the meaning of a robust probabilistic predictor and analyze how the restrictiveness and subjectiveness of the information set affect the robustness.
\end{problem}

Finally, we are interested in designing implementable robust probabilistic predictors with optimal performances.
\begin{problem}
Given different kinds of information sets, design robust probabilistic predictors and optimize their performances without violating the robustness.
\end{problem}

\begin{figure}[t]
   \centering
   \captionsetup{justification=centering}
   \includegraphics[width=\linewidth]{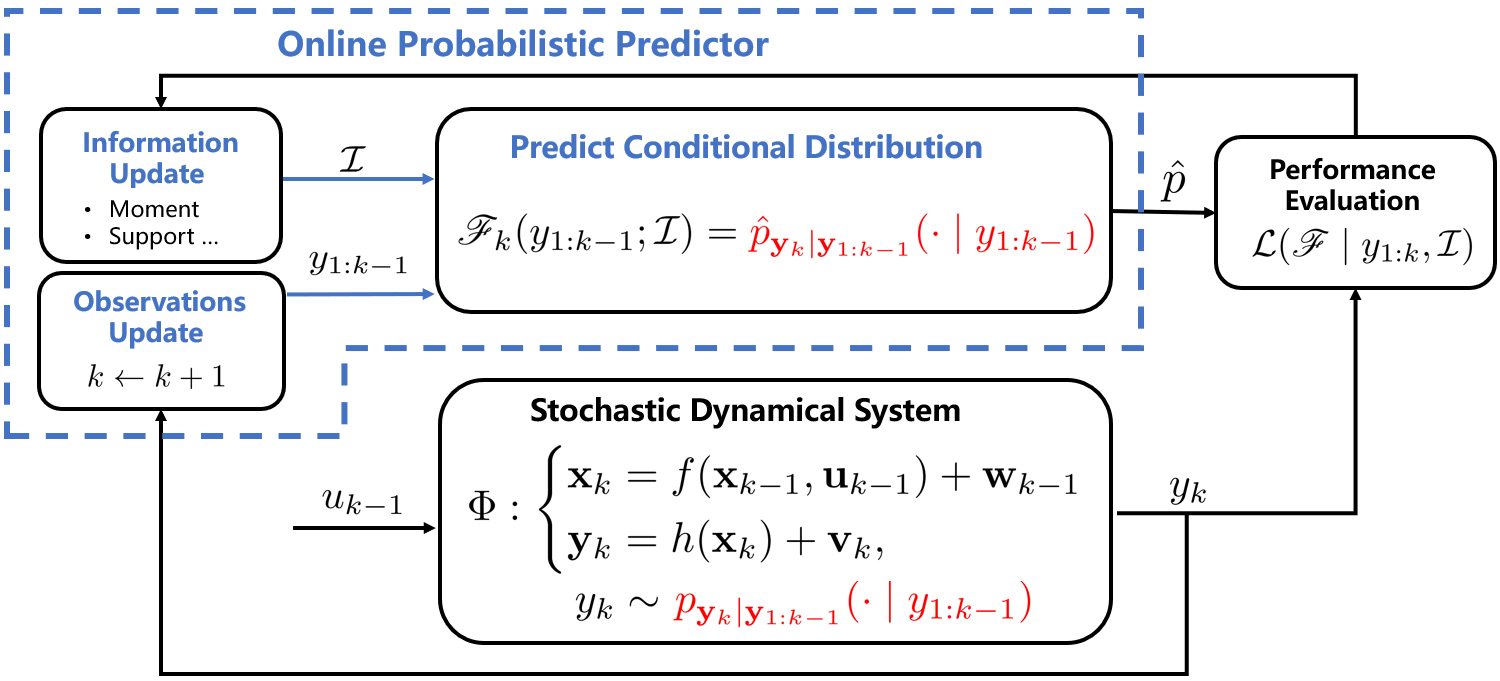}
   \caption{Online Probabilistic Predictor $\mathscr{F}$ for an SDS $\Phi$}
   \label{fig:problem}
\end{figure}

\section{Prediction Performance: Log-likelihood Functional Analysis}
In this section, we propose a metric to measure the prediction performance of a probabilistic predictor for SDS. To begin with, we define the metric by generalizing the idea of classical log-likelihood functions to log-likelihood functionals. Then, we provide a formal evaluation for the expected log-likelihood functional. Based on the evaluation, this metric is verified to be a proper scoring rule such that the metric is optimized when the predictive distributions equal the ground truth. It is also strictly proper in the Lebesgue measure sense. Nevertheless, we point out that an optimal performance is impossible when the predictor's information set is inadequate. Even worse, unrobust utilizing the information set is dangerous in unboundedly decreasing the performance. We provide an example to recognize this danger.

\subsection{Metric: Log-likelihood Functional}
When a $y_k$ is generated from the conditional distribution $p_{\mathbf{y}_k \mid \mathbf{y}_{1:k-1}}$, it is predicted by $\hat{p}_{\mathbf{y}_k \mid \mathbf{y}_{1:k-1}}$. To measure the prediction performance for this one-step probabilistic prediction, existing proper scoring rules are all theoretically acceptable. However, as the predictor's information set is inadequate, the real conditional distribution cannot be uniquely determined. Therefore, our metric should be local, i.e., it can evaluate the prediction performance based only through $\hat{p}_{\mathbf{y}_k \mid \mathbf{y}_{1:k-1}}$ and $y_k$ without knowing $p_{\mathbf{y}_k \mid \mathbf{y}_{1:k-1}}$. It can be shown that the log score is the only proper scoring rule that is local \cite{gneitingStrictlyProperScoring2007a,dawidTheoryApplicationsProper2014}, which characterizes the log-likelihood that $y_k$ is generated from $\hat{p}_{\mathbf{y}_k \mid \mathbf{y}_{1:k-1}}$. Naturally, we should define the prediction performance of a probabilistic predictor $\mathscr{F}$ on a trajectory as the likelihood that this trajectory can be generated from $\mathscr{F}$.

A standard likelihood function is of the form $\mathcal{L}(\theta | o_{1:n})$, where $o_{1:n}$ is the observations generated from some statistical model that can be parametrized by a vector $\theta$.
However, since the set containing all the online predictors is a functional space that may not be parametrized, we should generalize the idea of likelihood function to the likelihood functional.

\begin{definition}
   The log-likelihood functional of the online predictor $\mathscr{F}$ on a given trajectory of observations $y_{1:n}$ under the information set $\mathcal{I}$, is given as
   \begin{equation*}
      \begin{aligned}
         \mathcal{L}\left(\mathscr{F}, y_{1:n}\right)
          & :=\log \hat{p}_{\mathbf{y}_{1:n}}(y_1, \ldots, y_n)                                       \\
          & = \sum_{k=1}^{n} \log \hat{p}_{\mathbf{y}_k \mid \mathbf{y}_{1:k-1}}(y_k \mid y_{1:k-1}),
      \end{aligned}
   \end{equation*}
   where $\hat{p}_{\mathbf{y}_k \mid \mathbf{y}_{1:k-1}}(\cdot \mid y_{1:k-1}) = \mathscr{F}_k(y_{1:k-1}; \mathcal{I})$.
\end{definition}
\begin{remark}
   $\hat{p}_{\mathbf{y}_{1:n}}$ is the predicted joint probability density of $\mathbf{y}_{1:n}$, which can be decomposed to the product of one-step conditional probability densities based on the chain rule:
   \begin{equation}\label{eq:chain}
      \begin{aligned}
         \hat{p}_{\mathbf{y}_{1:n}}(y_1,\ldots, y_n)
         = & \; \hat{p}_{\mathbf{y}_{2:n} \mid \mathbf{y}_1}(y_2,\ldots, y_{n}\mid y_1)\hat{p}_{\mathbf{y}_{1}}(y_1) \\
         = & \prod_{k=1}^n \hat{p}_{\mathbf{y}_k \mid \mathbf{y}_{1:k-1}}(y_k \mid y_{1:k-1}).
      \end{aligned}
   \end{equation}
\end{remark}

The log-likelihood functional measures the online prediction performance of $\mathscr{F}$ on a specific trajectory of observations. To measure how well the predictor performs on all the other possible observation trajectories generated from the system, we need to study the expectation of the likelihood functional over the trajectories.
\begin{definition}
   The expected log-likelihood functional of the online predictor $\mathscr{F}$ under information set $\mathcal{I}$ is given as \[\mathcal{L}\left(\mathscr{F}, \mathbf{y}_{1:n} \right):=\E_{y_{1:n}} \mathcal{L}\left(\mathscr{F}, y_{1:n}\right).\]
\end{definition}
\begin{remark}
   From the perspective of statistical learning theory, $\mathcal{L}\left(\mathscr{F}, y_{1:n}\right)$ is similar to the concept of training loss because it reflects how well the data $y_{1:n}$ is consistent with the statistical model $\mathscr{F}$. Naturally, $\mathcal{L}\left(\mathscr{F}, \mathbf{y}_{1:n} \right)$ is similar to the concept of generalization loss, which is the expectation of the training loss on the data.
\end{remark}

Since the metric depends on the predictive distribution only through the realized outputs, it is a local scoring rule. Next, we should verify that the expected log-likelihood functional is indeed a proper scoring rule.

\subsection{Evaluation and Proper Scoring Rule}
Given a trajectory of observations $y_{1:n}$ and an online probabilistic predictor $\mathscr{F}$, evaluating the log-likelihood functional $\mathcal{L}\left(\mathscr{F}, y_{1:n}\right)$ is equivalent to evaluating the joint probability density function $\hat{p}_{\mathbf{y}_{1:n}}(y_1, \ldots, y_n)$. Nevertheless, even when the information set is adequate, this joint distribution does not have an analytical expression due to the nonlinear dynamics and non-Gaussian noises. When the information set is inadequate, evaluating the expected log-likelihood functional is more difficult. In the following theorem, we utilize the dynamics of the state-space model to derive a formal evaluation.

\begin{theorem}\label{thm:evl-1}
   The expected log-likelihood functional can be formally evaluated as follows,
   \begin{equation*}
      \mathcal{L}(\mathscr{F}, \mathbf{y}_{1:n})\!=\!-\mathrm{H}(\mathbf{y}_{1:n})- \sum_{k=1}^{n}\E_{y_{1:k-1}} D_{KL}\!\left(q_k\Vert \hat{q}_k\right),
   \end{equation*}
   where $q_k = p_{\mathbf{y}_k\mid \mathbf{y}_{1:k-1}}(\cdot\!\mid\!y_{1:k-1})$, $\hat{q}_k = \hat{p}_{\mathbf{y}_k\mid \mathbf{y}_{1:k-1}}(\cdot\!\mid\!y_{1:k-1})$.
\end{theorem}
\begin{proof}
   Please see Appendix \ref{app:thm:evl-1}.
\end{proof}

According to the non-negative property of KL-divergences, i.e.,
$D_{KL}\!\left(q_k\Vert \hat{q}_k\right) \leq 0$, we have the following corollary.
\begin{corollary}\label{cor:upper-bound-1}
   The expected log-likelihood functional is upper-bounded as follows,
   \begin{equation*}
      \mathcal{L}(\mathscr{F}, \mathbf{y}_{1:n}) \leq -\mathrm{H}(\mathbf{y}_{1:n}),
   \end{equation*}
   and the equality holds if and only if $p_{\mathbf{y}_k\mid \mathbf{y}_{1:k-1}} = \hat{p}_{\mathbf{y}_k\mid \mathbf{y}_{1:k-1}}$ in the sense of the Lebesgue measure.
\end{corollary}

Consistent with our intuition, corollary \ref{cor:upper-bound-1} confirms that $\mathcal{L}$ is a proper scoring rule. Furthermore, it is strictly proper in the sense of the Lebesgue measure.

\subsection{Optimal Performance Is Impossible and The Danger of Unrobust Predictor}
The optimality in corollary \ref{cor:upper-bound-1} is nearly impossible to attain for a predictor with an inadequate information set. We explain how optimality is prevented from three perspectives.

First, because the information set is inadequate, it is impossible to uniquely determine the conditional distributions. Furthermore, the space of all the feasible distributions is very large and complex.

Second, as the prediction step increases, the predictor may be able to infer the distribution of system noises and control inputs based on previous observations. However, according to the No Free Lunch theorem \cite{adamNoFreeLunch2019}, it is impossible for any learning algorithm to accurately learn the real distribution when there is too little data. At the beginning of an online prediction task, the observations are too few to support efficient learning.

Third, even if the system noise and control inputs are somehow learned by the predictor, the system controller may adversarially adjust the design of control inputs to degrade the prediction performance without violating the information set.

\begin{example}
   Consider a one-dimensional stochastic dynamical system with noiseless observation,
   \begin{equation}\label{eq:ex-one-dim}
      \left\{
      \begin{aligned}
          & \mathbf{x}_{k+1} = \mathbf{x}_k + \mathbf{u}_k + \mathbf{w}_k \\
          & \mathbf{y}_k     = \mathbf{x}_k,
      \end{aligned}
      \right.
   \end{equation}
   where the information set $\mathcal{I}$ contains:
   \begin{enumerate}
      \item $\E\mathbf{u}_k = \mu_k$,
      \item $\E \mathbf{w}_k = 0$,
      \item $\operatorname{supp}(\mathbf{y}_k) = \R^{d_y}$.
   \end{enumerate}
   Suppose the predictor use Gaussian distribution to predict the system \eqref{eq:ex-one-dim}, specifically \[\hat{p}_{\mathbf{y}_{k+1}\mid \mathbf{y}_{1:k}} \sim \mathcal{N}(\mathbf{y}_k + \mu_k, \sigma_k^2),\] where $\sigma_k > 0$ is an adjustable hyperparameter.
\end{example}
If the second order moment of $\mathbf{u}_k + \mathbf{w}_k$ exists, i.e., $\tilde{\sigma}^2_k = \operatorname{Var}(\mathbf{u}_k + \mathbf{w}_k)$, we can calculate the one-step log-likelihood functional as follows,
\begin{equation*}
   \begin{aligned}
      \E \log \hat{p}_{\mathbf{y}_{k+1}\mid \mathbf{y}_{1:k}} \!= & \int p_{\mathbf{y}_{k+1}\mid \mathbf{y}_{1:k}}(s\!\mid\! z) \log\!\left(\frac{e^{-\frac{(s-y_k-\mu_k)^2}{2\sigma^2_k}}}{\sqrt{2\pi}\sigma_k}\!\right)\!\dd s\dd z \\
      =                                                           & -\int p_{\mathbf{y}_{k+1}\mid \mathbf{y}_{1:k}}(s\!\mid\! z) \log(\sqrt{2\pi}\sigma_k)\dd s\dd z                                                                  \\
                                                                  & -\frac{1}{2\sigma^2_k}\!\int\! p_{\mathbf{y}_{k+1}\mid \mathbf{y}_{1:k}}(s\!\mid\! z)(s\!-\!y_k\!-\!\mu_k)^2 \dd s\dd z                                           \\
      =                                                           & -\log(\sqrt{2\pi}\sigma_k) - \frac{\tilde{\sigma}^2_k}{2\sigma_k^2}.
   \end{aligned}
\end{equation*}
However, the first-order information set $\mathcal{I}_1$ allows the case that $\tilde{\sigma}^2_k = \infty$, (e.g., a Student's t-distribution with order $2$ has finite expectation but infinite variance). Towards this Gaussian-based predictor, the controller can properly adjust the distribution of $\mathbf{u}$ such that $p_{\mathbf{y}_{k+1}\mid \mathbf{y}_{1:k}}$ is a Student's t-distribution with order $2$, then the prediction performance will be unboundedly decreased such that \[\mathcal{L}(\mathscr{F}, \mathbf{y}_{1:n}) = -\infty.\]

The above example illustrates that \textit{an unrobust predictor faces the danger of unbounded performance decrease.} Conversely, a robust predictor under $\mathcal{I}$ should be able to ensure that the performance is bounded below no matter how the control inputs are designed.

In summary, $\mathcal{L}(\mathscr{F}, y_{1:n})$ is a proper local scoring rule for probabilistic prediction, and the optimal performance can be attained when the information set is adequate. However, an inadequate information set makes it impossible for the predictor to achieve optimal performance. An adversarial controller is even capable of unboundedly decreasing the prediction performance when the predictor is unrobust.

\section{Robust Probabilistic Prediction}
In this section, we propose a functional-optimization-based framework to quantitatively specify the meaning of robustness for probabilistic predictors. Then we focus on the moment-based information sets and design a class of moment-based robust probabilistic predictors. Furthermore, we optimize the performances of these robust predictors concerning moment-based information sets that are of different orders.

\subsection{Robust Probabilistic Prediction Framework}
Now that the information provided by $\mathcal{I}$ is insufficient to uniquely determine $p_{\mathbf{y}_k\mid \mathbf{y}_{1:k-1}}$, the predicted $\hat{p}_{\mathbf{y}_k\mid \mathbf{y}_{1:k-1}}$ may significantly deviate. Moreover, the system controller may adversarially change its distribution without violating the constraints of the predictor's information set. A robust probabilistic predictor should exploit the information set to ensure that the worst-case prediction performance will not be significantly degraded.

\begin{definition}[Robust probabilistic predictor]
   A probabilistic predictor $\mathscr{F}$ for SDS $\Phi$ with an information set $\mathcal{I}$ is robust if the worst-case prediction performance is lower bounded, i.e., \[ \min\limits_{\mathbf{u}} \mathcal{L}(\mathscr{F}, \mathbf{y}_{1:n}) > -\infty.\]
\end{definition}
This definition quantitatively specifies the meaning of the worst case by an optimization-based framework: the minimum prediction performance should not be unboundedly decreased no matter how the controller designs inputs under the constraints of system dynamics and information set.

Although a robust probabilistic predictor ensures the existence of the performance lower bound, it may be conservative. Therefore, we are interested in deriving the optimal robust probabilistic predictor, which is the maximizer of the following max-min problem:
\begin{align}\label{eq:maxmin}
                & \max\limits_{\mathscr{F}}\min\limits_{\mathbf{u}} \mathcal{L}(\mathscr{F}, \mathbf{y}_{1:n}) \nonumber \\
   \text{s.t. } & \left\{
   \begin{aligned}
       & \mathbf{x}_{k+1}  = f(\mathbf{x}_k, \mathbf{u}_k) + \mathbf{w}_k        \\
       & \mathbf{y}_k      = h(\mathbf{x}_k) + \mathbf{v}_k,    \;1\leq k \leq n \\
       & \mathcal{I}.
   \end{aligned}
   \right.
\end{align}

The above max-min functional optimization problem can be described as a dynamic game between the predictor and the controller: the predictor tries to maximize the prediction performance, while the controller aims to minimize it; $\mathcal{I}$ not only serves as the information for the predictor but also as a limitation to the controller. At each step, the predictor is challenged with predicting the output which is affected by the active input designs from the controller.

In summary, our framework shows that to analyze whether a probabilistic predictor is robust is to solve an information-set-constrained functional optimization problem; to derive an optimal robust probabilistic predictor is to solve a max-min functional optimization problem.

\subsection{Moment-based Robust Probabilistic Predictor}
In practice, our prior knowledge of a stochastic dynamical system is usually about the moment information, such as the expectations and covariances of the noises. Most of the time, the available information set $\mathcal{I}$ is very restrictive such that only the low-order moments are known. Even worse, $\mathcal{I}$ may only guarantee the existence of some low-order moments rather than exactly knowing their values. We summarize this type of information as the following moment-based information set.

\begin{definition}
   The $m$-th moment information set for the stochastic dynamical system $\Phi$ is
   \[\mathcal{I}_m\!:=\!\left\{|\mu_\alpha(\mathbf{v})| \!<\! \infty , \forall \mathbf{v}\in \{\mathbf{w}_{1:n}, \mathbf{v}_{1:n}, \mathbf{u}_{1:n}, \mathbf{x}_0\}, |\alpha|\leq m\right\}.\]
\end{definition}

Concerning information sets $\{\mathcal{I}_m\}_{m=0}^2$, we will design a class of robust probabilistic predictors called the moment-based robust probabilistic predictor, and solve out the optimal robust probabilistic predictor. Then we analyze how the contents of the information set influence the performance.

Designing robust online predictors is equivalent to finding sufficient conditions for $\mathscr{F}$ to ensure $\min\limits_{\mathbf{u}}\mathcal{L}(\mathscr{F}, \mathbf{y}_{1:n}) > -\infty$.
A natural idea is to explicitly derive the minimal likelihood with a fixed predictor $\mathscr{F}$, then analyze which kinds of predictors will result in a $-\infty$ minimal likelihood. However, this method is too ideal to be realized, because it requires complete information about the SDS to minimize over $\mathbf{u}_{0:n-1}$ under the constraint that their $m$-th order moments exist.

Rather than the natural idea of solving the minimizer first, our method is to study when the likelihood will be negative infinite without solving the minimizer.
\begin{lemma}\label{lem:finite-condition}
   That the minimum expected log-likelihood being lower bounded is equivalent to that each one-step log-likelihood is lower bounded, i.e., for $1\leq k \leq n$, there is
   \begin{equation*}
      \begin{aligned}
          & \min\limits_{p}  \!\!\int\! p_{\mathbf{y}_k \!\mid \mathbf{y}_{1:k-1}}(s\!\mid\! {y}_{1:k-1})\log \hat{p}_{\mathbf{y}_k \mid \mathbf{y}_{1:k-1}}(s\!\mid\! {y}_{1:k-1}) \dd s > -\infty
      \end{aligned}
   \end{equation*}
   holds almost everywhere for trajectory $y_{1:k-1}$.
\end{lemma}
\begin{proof}
   Please see Appendix \ref{app:lem:finite-condition}.
\end{proof}

Lemma \ref{lem:finite-condition} reduces the problem of designing robust $\mathscr{F}$ to designing robust one-step predictor $\mathscr{F}_k$. Specifically, we need to figure out what kinds of $\hat{p}_{\mathbf{y}_k \mid \mathbf{y}_{1:k-1}}$ guarantee the log-likelihood being lower bounded when $p_{\mathbf{y}_k \mid \mathbf{y}_{1:k-1}}$ is constraint by the information set $\mathcal{I}_m$.

The next problem is, the information set $\mathcal{I}_m$ is not directly related to the conditional distribution $p_{\mathbf{y}_k \mid \mathbf{y}_{1:k-1}}$, and there is no simple rule to characterize how much statistical information of $p_{\mathbf{y}_k \mid \mathbf{y}_{1:k-1}}$ can be exploited from $\mathcal{I}_m$. A basic conclusion can be drawn from $\mathcal{I}_m$ is provided as follows.
\begin{lemma}
   The information set $\mathcal{I}_m$ implies that the moment of $p_{\mathbf{y}_k \mid \mathbf{y}_{1:k-1}}$ with order not greater than $m$ exists.
\end{lemma}
This lemma can be easily proved since $\mathcal{I}_m$ indicates the $m$-th order of moment of $p_{\mathbf{y}_k}$ exists, which further indicates that the $m$-th order of moment of the conditional distribution $p_{\mathbf{y}_k \mid \mathbf{y}_{1:k-1}}$ exists. We utilize the above two lemmas by considering an auxiliary optimization problem.

\begin{align}\label{eq:maxmin-aux}
    & \min\limits_{p}  \int\! p_{\mathbf{y}_k \!\mid \mathbf{y}_{1:k-1}}(s\!\mid\! y_{1:k-1})\log \hat{p}_{\mathbf{y}_k \mid \mathbf{y}_{1:k-1}}(s\!\mid\! y_{1:k-1}) \dd s \nonumber \\
    & \text{s.t. }      \!\!                    \left\{
   \begin{aligned}
       & \mathbf{x}_{i+1}  = f(\mathbf{x}_i) + g(\mathbf{x}_i)\mathbf{u}_i + \mathbf{w}_i                                                                         \\
       & \mathbf{y}_i      = h(\mathbf{x}_i) + \mathbf{v}_i, \; 1\leq i \leq k                                                                                    \\
       & \mu_\alpha(\mathbf{y}_k \mid y_{1:k-1})    \! = \!\!\int\!\! s^\alpha p_{\mathbf{y}_k \mid \mathbf{y}_{1:k-1}}(s\!\mid\! y_{1:k-1})\dd s, |\alpha|\leq m \\
       & \operatorname{supp}(\mathbf{y}_k),
   \end{aligned}
   \right.
\end{align}
where
\begin{equation}
   \operatorname{supp}(\mathbf{y}_k) \!=\! \{v \in \R^{d} \!\mid\! \underline{y}_k^{(i)} \!\leq\! v^{(i)}\!\leq\! \bar{y}_k^{(i)}, 1\leq i \leq d_y\},
\end{equation}
and both $\underline{y}_k^{(i)}$ and $\bar{y}_k^{(i)}$ belong to $\R\cup\{-\infty,+\infty\}$. This auxiliary problem follows by replacing $\mathcal{L}(\mathscr{F}, \mathbf{y}_{1:n})$ and $\mathcal{I}_m$ by one-step log-likelihood and conditional moments respectively. The rationale of this replacement is guaranteed by Lemma \ref{lem:aux}.

\begin{lemma}\label{lem:aux}
   If the minimum value of \eqref{eq:maxmin-aux} is bounded below, the predictor $\mathscr{F}$ such that $\mathscr{F}_k(y_{1:k-1}; \mathcal{I}_m) = \hat{p}_{\mathbf{y}_k \mid \mathbf{y}_{1:k-1}}(\cdot \mid y_{1:k-1})$ is robust.
\end{lemma}
\begin{proof}
   Since the moment constraints of $p$ are necessary conditions exploited from $\mathcal{I}_m$, it follows that any $p_{\mathbf{y}_k \!\mid \mathbf{y}_{1:k-1}}$ that is feasible for the original problem is also feasible to this problem. In other words, the feasible space is enlarged. Therefore, the minimum value to this problem is no larger than the original problem, any $\hat{p}$ ensuring the existence of a finite lower bound immediately guarantees that it is a robust probabilistic predictor.
\end{proof}
\begin{remark}
   The robust probabilistic predictors satisfying this lemma belong to a subset of all the robust probabilistic predictors since the constraints only contain the moment-based information on $p_{\mathbf{y}_k \mid \mathbf{y}_{1:k-1}}$, which is a subset of $\mathcal{I}_m$. In fact, any other information exploited from $\mathcal{I}_m$ can be added as constraints in the optimization problem, our choice of the moment-based information provides only one way to design robust probabilistic predictors.
\end{remark}

In the next theorem, we present a sufficient and necessary condition on the existence of a lower-bounded optimal value for the auxiliary problem \eqref{eq:maxmin-aux}.
\begin{theorem}[$m$-th moment robust probabilistic predictor]\label{thm:robust-form}
   An probabilistic predictor $\mathscr{F}$ is robust under $\mathcal{I}_m$ if it has a polynomial-exponential form such that
   \begin{equation*}
      \mathscr{F}_k(y_{1:k-1}; \mathcal{I}_m) = e^{\sum_{i=0}^{m} \sum_{|\alpha|=i} \lambda_\alpha s^\alpha }
   \end{equation*}
   where $\lambda_\alpha \in \R$.
\end{theorem}
\begin{proof}
   Please see Appendix \ref{app:thm:robust-form}
\end{proof}

To have an intuitive understanding of why the $m$-th order robust probabilistic predictor belongs to the exponential-polynomial families whose polynomial order is less than $m$, we provide a heuristic explanation as follows. Suppose that $\log \hat{p}_{\mathbf{y}_k \mid \mathbf{y}_{1:k-1}}$ can be expanded as a multivariate Taylor series, such that \[\log \hat{p}_{\mathbf{y}_k \mid \mathbf{y}_{1:k-1}}(s\mid y_{1:k-1}) = \sum_{i=0}^{m} \sum_{|\alpha|=i} \lambda_\alpha s^\alpha .\]
It follows that
\begin{equation*}
   \begin{aligned}
        & \int p_{\mathbf{y}_k \!\mid \mathbf{y}_{1:k-1}}(s\!\mid\! y_{1:k-1})\log \hat{p}_{\mathbf{y}_k \mid \mathbf{y}_{1:k-1}}(s\!\mid\! y_{1:k-1}) \dd s \\
      = & \sum_{i=0}^{\infty} \sum_{|\alpha|=i} \lambda_\alpha \mu_\alpha(\mathbf{y}_k \mid y_{1:k-1}).
   \end{aligned}
\end{equation*}

If $\lambda_\alpha = 0$ for any $\alpha$ such that $|\alpha| > m$, the objective is bounded below. Otherwise, suppose there exist at least one $\beta$ with $|\beta|> m$ such that $\lambda_\beta \neq 0$. By letting $p$ subject to a multivariate Student's t-distribution with order $|\beta|$, the objective becomes negative infinite. Therefore, the order of $\log \hat{p}$ should be no more than $m$.

\subsection{Optimal m-th Moment Robust Probabilistic Predictor}
For each $\mathcal{I}_m$, there are many feasible $m$-th moment robust probabilistic predictors. To improve their performance, we are interested in tuning the parameters $\lambda_\alpha$ in Theorem \ref{thm:robust-form} to maximize the performance.

\subsubsection{Information Set of Zeroth Order}
$\mathcal{I}_0$ describes the situation where the predictor is not confident with any statistical feature of the system. Intuitively, when the expectation of control inputs to an SDS is not guaranteed to exist, making a robust probabilistic prediction is hard. The following lemma shows that a zeorth moment robust probabilistic predictor exists if and only if the information on the support can be limited to a bounded set.
\begin{lemma}
   A zeroth moment robust probabilistic predictor exists if and only if the support of $\mathbf{y}_{k}$ where $1\leq k \leq n$ is both upper and lower bounded elementwise, i.e., $\underline{y}_k^{(i)} > -\infty$ and $\bar{y}_k^{(i)} < \infty$ hold for $1\leq i \leq d_y$.
\end{lemma}
\begin{proof}
   When $m=0$, there is $\hat{p}_{\mathbf{y}_k \mid \mathbf{y}_{1:k-1}}(s \mid y_{1:k-1}) = e^{\lambda_0}$, which is the probability density function of a uniform distribution. If the support of $\mathbf{y}_k$ is not both upper and lower bounded, the uniform distribution cannot be defined. It can be concluded that there is no zeroth moment robust probabilistic predictor if $\operatorname{supp}(\mathbf{y}_k)$ is not both upper and lower bounded.
\end{proof}

Once $\underline{y}_k^{(i)} > -\infty$ and $\bar{y}_k^{(i)} < \infty$ hold for $1\leq i \leq d_y$, deriving optimal zeroth moment robust probabilistic predictor is trivial as follows.
\begin{theorem}\label{thm:maxmin-0}
   The output of the optimal zeorth moment robust probabilistic predictor $\mathscr{F}_k^\star$ is a uniform distribution such that
   \begin{equation*}
      \hat{p}_{\mathbf{y}_k \mid \mathbf{y}_{1:k-1}}^\star(s\mid y_{1:k-1}) \!\!=\!\left\{ \!\!\!\begin{array}{cl}
         \prod_{i=1}^{d_y}\left(\bar{y}_k^{(i)}-\underline{y}_k^{(i)}\right)^{-1} & \!s\! \in \operatorname{supp}(\mathbf{y}_k) \\
         0                                                                        & \text{else.}
      \end{array}\right.
   \end{equation*}
\end{theorem}

\subsubsection{Information Set of Second Order} Contrary to the zeroth order information set, the statistical information contained in $\mathcal{I}_2$ is much more abundant. When the second moments are utilized, second-moment robust probabilistic predictors always exist even if there is no prior knowledge of the supports. For the ease of writing, we make a few simplifications to the notations. In the rest of this section, we use
\begin{equation}
   \left\{\begin{aligned}
       & \Sigma \!=\! \mu_2(\mathbf{y}_k \!\mid\! y_{1:k-1}) \!-\! \mu_1(\mathbf{y}_k \!\mid\! y_{1:k-1})\mu_1(\mathbf{y}_k \!\mid\! y_{1:k-1})^\top \\
       & z = \mu_1(\mathbf{y}_k \mid y_{1:k-1}).
   \end{aligned}\right.
\end{equation}

\begin{theorem}\label{thm:maxmin-2}
   If the information set is of the second order, a robust probabilistic predictor exists. The output of the optimal second-moment robust probabilistic predictor $\mathscr{F}_k^\star$ is a Gaussian distribution, such that
   \begin{align*}
      \hat{p}_{\mathbf{y}_{k} \mid \mathbf{y}_{1:k-1}}^\star(s \mid y_{1:k-1})= & (2 \pi)^{-{d_y\over 2}} \operatorname{det}(\Sigma)^{-{1\over 2}}   \\
         & \times e^{\left(-\frac{1}{2}(s-z)^{\top} \Sigma^{-1}(s-z)\right)}.
   \end{align*}
\end{theorem}
\begin{proof}
   Please see Appendix \ref{app:thm:maxmin-2}.
\end{proof}

Compared to the zeroth-order moment robust probabilistic predictor, the optimal second-order moment robust probabilistic predictor improved the prediction performance prominently. As shown in the proof of Theorem \ref{thm:maxmin-2}, given the previous trajectory $y_{1:k-1}$, the maximal one-step log-likelihood at time $k$ is \[-\frac{d_y}{2}\log(2\pi) \!-\! \frac{1}{2}\log\operatorname{det}(\Sigma) \!-\! \frac{1}{2}\langle \Sigma,\Sigma^{-1}\rangle.\]

It should be noted that attaining optimality requires accurate estimations of $z$ and $\Sigma$. When the estimations are not accurate, Appendix \ref{app:thm:maxmin-2} shows that the prediction performance is \[-\frac{d_y}{2}\log(2\pi) \!-\! \frac{1}{2}\log\operatorname{det}(\hat{\Sigma}) \!-\! \frac{1}{2}\langle \Sigma,\hat{\Sigma}^{-1}\rangle \!-\! \|z \!-\! \hat{z}\|_{\hat{\Sigma}^{-1}}^2.\] Therefore, a too-small covariance (in the sense of determinant) and an inaccurate estimation of expectation can lead to performance decreasing polynomially fast.

Practically, real-world noises usually possess a heavy-tailed feature, where the second-order moments may not exist. Unless a strong guarantee is provided on the existence of second-order moments, it is generally too ideal to use a robust probabilistic predictor of second or even higher moments.

\subsubsection{Information Set of First Order} $\mathcal{I}_1$ characterizes one of the most common situations in practice where the predictor is confident that the expectation of the input, process noises and observation noises exist, but doubts the existence of their covariances (or cannot have confident estimations of the covariances). The example in the previous section has shown the bad effect of an unrobust predictor that inappropriately uses Gaussian distribution to predict a long-tail distribution.

Like the zeroth moment robust probabilistic predictor, the existence of a first moment robust probabilistic predictor requires additional information of $\operatorname{supp}(\mathbf{y}_k)$.
\begin{lemma}
   If the information set is of the first order, a first-moment robust probabilistic predictor exists if and only if the support of $\mathbf{y}_{k}$ where $1\leq k \leq n$ is not both upper and lower unbounded.
\end{lemma}
\begin{proof}
   When $m=1$, there is $\hat{p}_{\mathbf{y}_k \mid \mathbf{y}_{1:k-1}}(s \mid y_{1:k-1}) = e^{\lambda_1^\top s + \lambda_0^\top \mathbf{1}}$ with $\lambda_0, \lambda_1 \in \R^{d_y}$, which is a multivariate exponential distribution. If $\operatorname{supp}(\mathbf{y}_k) = (-\infty, \infty)$, the exponential distribution cannot be defined. It can be concluded that there is no first-moment robust probabilistic predictor if $\operatorname{supp}(\mathbf{y}_k)$ is not both upper and lower unbounded.
\end{proof}
The next theorem shows that a first-moment robust probabilistic predictor should use exponential distribution based on the knowledge of the first-order moments.
Let $x^*$ be the non-zero solution to the equation \[ [(\bar{y}_k^{(i)} \!-\! z^{(i)})x \!-\! 1]e^{(\bar{y}_k^{(i)} \!-\! \underline{y}_k^{(i)})x} \!-\! (\underline{y}_k^{(i)} \!-\! z^{(i)})x \!+\! 1 = 0, \] where $z = \mu_1(\mathbf{y}_k \mid y_{1:k-1})$. We define the following notations:
\begin{equation*}
   \begin{aligned}
      \lambda_1^{(i)} & \!=\!\left\{ \begin{array}{lll}
                                         & \left(\underline{y}_k^{(i)} - z^{(i)}\right)^{-1} & \bar{y}_k^{(i)} = \infty        \\
                                         & \left(\bar{y}_k^{(i)} - z^{(i)}\right)^{-1}       & \underline{y}_k^{(i)} = -\infty \\
                                         & x^*                                               & \text{else},
                                     \end{array}\right.                                                                     \\
      \lambda_0^{(i)} & \!=\!\left\{ \begin{array}{ll}
                                        -\log\left(z^{(i)} - \underline{y}_k^{(i)}\right) + \frac{\underline{y}_k^{(i)}}{z^{(i)} - \underline{y}_k^{(i)}}    & \bar{y}_k^{(i)} \!=\! \infty        \\
                                        -\log\left(\bar{y}_k^{(i)} - z^{(i)}\right) + \frac{\bar{y}_k^{(i)}}{z^{(i)} - \bar{y}_k^{(i)}}                      & \underline{y}_k^{(i)} \!=\! -\infty \\
                                        -\log\left(e^{\bar{y}_k^{(i)}\lambda_1^{(i)}} - e^{\underline{y}_k^{(i)}\lambda_1^{(i)}}\right) +\log\lambda_1^{(i)} & \text{else}.
                                     \end{array}\right.
   \end{aligned}
\end{equation*}

\begin{theorem}\label{thm:maxmin-1}
   The output of the optimal first-moment robust probabilistic predictor $\mathscr{F}_k^\star$ is an exponential distribution,
   \begin{equation*}
      \hat{p}_{\mathbf{y}_k \mid \mathbf{y}_{1:k-1}}^\star(s\mid y_{1:k-1}) \!=\!\! \left\{\!\! \begin{array}{cl}
         \prod_{i=1}^{d_y} e^{\lambda_1^{(i)}s^{(i)} + \lambda_0^{(i)}} & s \in \operatorname{supp}(\mathbf{y}_k) \\
         0                                                              & \text{else}.
      \end{array}\right.
   \end{equation*}
\end{theorem}
\begin{proof}
   Please see Appendix \ref{app:thm:maxmin-1}.
\end{proof}

Theorem \ref{thm:maxmin-1} implies that the first-order moment robust probabilistic predictor has two limitations. First, designing a first-order moment robust probabilistic predictor requires the support of each dimension to be at least half-bounded. This inherent limitation of the polynomial moment-based robust probabilistic predictor makes it incapable of handling situations where no support information is available. Second, the optimal performance of the predictor is attained if and only if $z$ can be accurately estimated by the predictor. Moreover, this performance may be highly sensitive to the estimation of $z$ as analyzed in the following example.
\begin{example}
   Consider a one-dimensional first-order moment robust probabilistic predictor at time step $k$ with $\operatorname{supp}(\mathbf{y}_k) = [\underline{y}_k, \infty)$, there is
   \begin{equation*}
      \begin{aligned}
          & \frac{e^{\lambda_0}}{\lambda_1}(0 - e^{\lambda_1\underline{y}_k}) = 1 \Rightarrow \lambda_0 = \log(-\lambda_1) - \lambda_1\underline{y}_k.
      \end{aligned}
   \end{equation*}
   Using this equality, the one-step log-likelihood follows as
   \begin{equation*}
      \begin{aligned}
          & \lambda_1 z + \lambda_0 = \lambda_1(z - \underline{y}_k) + \log(-\lambda_1),
      \end{aligned}
   \end{equation*}
   which is maximized when $\lambda_1 = (\underline{y}_k - z)^{-1}$. However, the function \[\Gamma(x) := x(z - \underline{y}_k) + \log(-x)\] approaches $-\infty$ exponentially fast around $0$. It implies that an inaccurate estimation of $z$ risks significantly degrading the performance of a first-order moment robust probabilistic predictor.
\end{example}

\subsection{Non-polynomial Moment Robust Probabilistic Predictor}
The polynomial moment robust probabilistic predictor is derived from the optimization problem \ref{eq:maxmin-aux}, where the moment $\mu_\alpha(\mathbf{y}_k \mid y_{1:k-1})$ is utilized to constrain the conditional distribution $p_{\mathbf{y}_k \mid \mathbf{y}_{1:k-1}}$. Specifically, the first-moment robust probabilistic predictor uses $z$, and it suffers from the limitations of support knowledge and instability.
To mediate these two limitations, a natural idea is to exploit more statistical features from $\mathcal{I}_1$ and use them to further constrain $p_{\mathbf{y}_k \mid \mathbf{y}_{1:k-1}}$.

There is some non-polynomial moment information about $p_{\mathbf{y}_k \mid \mathbf{y}_{1:k-1}}$ can be exploited from $\mathcal{I}_1$, as shown in the following lemma.
\begin{lemma}\label{lem:nonpoly-moment}
   Given a first order information set, we have

   (i) the first-order absolute value moment exists, i.e., \[\int |s|\, p_{\mathbf{y}_k \mid \mathbf{y}_{1:k-1}}(s\mid y_{1:k-1}) \dd s < \infty\]

   (ii) the quadratic logarithm moment exists, i.e., \[\int \log(1 + s^\top s)\, p_{\mathbf{y}_k \mid \mathbf{y}_{1:k-1}}(s\mid y_{1:k-1}) \dd s < \infty.\]
\end{lemma}
\begin{proof}
   Please see Appendix \ref{app:lem:nonpoly-moment}.
\end{proof}

Similar to the design of polynomial moment-based predictors, a first-order absolute value moment predictor can defined such that the one-step predictor equals a multivariate Laplace distribution, i.e.,
\[
   \hat{p}_{\mathbf{y}_k \mid \mathbf{y}_{1:k-1}}(s\mid y_{1:k-1})=\prod_{i=1}^{d_y} \frac{\lambda_1^{(i)}}{2}e^{-\lambda_1^{(i)}|s^{(i)} - \lambda_0^{(i)}|},
\]
where $\lambda_1^{(i)} < 0$ and the support is unbounded.

Then, we define a quadratic logarithm moment predictor such that $\hat{p}_{\mathbf{y}_k \mid \mathbf{y}_{1:k-1}}(s\mid y_{1:k-1})$ equals a multivariate t distribution:
$$
   \frac{\Gamma[(\nu+d_y) / 2]}{\Gamma(\nu / 2) \nu^{d_y / 2} \pi^{d_y / 2}|{\Sigma}|^{1 / 2}}\left[1\!+\!\frac{1}{\nu}(s\!-\!\lambda_1)^T {\Sigma}^{-1}(s\!-\!\lambda_1)\right]^{\frac{-(\nu\!+\!d_y)}{2}},
$$
where $\nu > 0$.

Based on Lemma \ref{lem:nonpoly-moment}, it follows immediately that both these two predictors are robust.
\begin{theorem}
   Given $\mathcal{I}_1$, the first-order absolute value moment probabilistic predictor and the quadratic logarithm moment probabilistic predictor are robust. 
\end{theorem}

Compared to the first-order polynomial moment probabilistic predictor, the first-order absolute value moment probabilistic predictor does not require prior knowledge of the support, and the prediction performance is maximized when \[\int_{-\infty}^{\lambda_0^{(i)}} {p}_{\mathbf{y}_k \mid \mathbf{y}_{1:k-1}}(s\mid y_{1:k-1}) \dd s^{(i)} = 0.5,\] and $\lambda_1^{(i)} =\int p_{\mathbf{y}_k \mid \mathbf{y}_{1:k-1}}(s\mid y_{1:k-1})\, |s - \lambda_0| \dd s$.

The quadratic logarithm moment probabilistic predictor, which is more conservative than the first-order absolute value moment predictor, also does not require prior knowledge of the support. If $p_{\mathbf{y}_k \mid \mathbf{y}_{1:k-1}}(\cdot\mid y_{1:k-1})$ is an even function, the prediction performance is maximized when $\lambda_1 = z$. Moreover, its prediction performance is less sensitive to the parameters $\lambda_1$, which means an inaccurate estimation of the expectation will not decrease the performance too much as the first-order polynomial moment predictor does.

\subsection{Summary and Discussion}
A robust probabilistic predictor indicates a restrained and wise utilization of the information about the SDS, hence the prediction performance will never be unboundedly decreased. Focusing on the moment information, we have found a class of moment-based probabilistic predictors to be robust. The more accurate predictors estimate the true conditional moments, the better prediction performance is. Since polynomial moment probabilistic predictors may require extra support information and are sensitive to estimation accuracy, non-polynomial moment predictors are introduced to mediate these problems.

Certainly, the moment-based robust probabilistic predictor is not the only type. There are many different methods to exploit other statistical features of $p_{\mathbf{y}_k \mid \mathbf{y}_{1:k-1}}$ from $\mathcal{I}_m$. Just like our generalization from polynomial moment-based predictors to non-polynomial ones, different statistical features will yield different forms of robust probabilistic predictors. However, a new robust probabilistic predictor is necessary only when it possesses better performance or robustness, e.g., non-polynomial moment predictors are less sensitive to support and estimation accuracy.

\section{Algorithm Implementation}
In practice, a complete design of a moment-based robust probabilistic predictor can be divided into three steps: i) determine the order of the information set; ii) exploit the information set $\mathcal{I}_m$ and previous observations $\mathbf{y}_{1:k-1}$ to estimate the conditional moments $\mu_\alpha(\mathbf{y}_k \mid y_{1:k-1})$ at each time step for multi-index $\alpha$ with $|\alpha| \leq m$; iii) solve the optimal parameters of the moment-based robust probabilistic predictor following the estimations of $\mu_\alpha$.

In the last section, we already know how to handle step $3$ when the estimations of $\mu_\alpha$ are available, now we should handle the second step. It should be noted that the estimation in the second step may not be accurate. Since $p_{\mathbf{y}_k \mid \mathbf{y}_{1:k-1}}(s\mid y_{1:k-1})$ is unknown, $\mu_\alpha(\mathbf{y}_k \mid y_{1:k-1})$ cannot be directly derived from the definition. Instead, it should be approximated (e.g., Kalman filter) or learned (e.g., RNN), wherein estimation errors are unavoidable. That's why our design of robust probabilistic predictors should be robust with the unknown distributions but also robust/insensitive with the estimation accuracy.

In this section, we revisit the Kalman filter from the perspective of online predictors and leverage it as a realization of step $2$, to design moment-based robust probabilistic predictors for a linear stochastic dynamical system:
\begin{equation}\label{eq:linear-sds}
   \left\{\begin{aligned}
       & \mathbf{x}_{k+1} = F_k\mathbf{x}_k + G_k\mathbf{u}_k + \mathbf{w}_k \\
       & \mathbf{y}_k     = H_k\mathbf{x}_k + \mathbf{v}_k,
   \end{aligned}\right.
\end{equation}
where the information set includes the value of $F_k, G_k, H_k$.

\subsection{Revisit Kalman Filter As An Online Point Predictor}
Kalman filter is one of the most classical algorithms for online point prediction, which recursively fuses the innovations from observations into system models to update estimations. Let $\mathbf{\hat{x}}_{k}^{-}$, $\hat{P}_{k}^{-}$ be the prediction and covariance of $\mathbf{x}_{k}$ before $\mathbf{y}_{k}$ is observed, and $\mathbf{\hat{x}}_{k}^{+}$ ,$\hat{P}_{k}^{+}$ be the estimation and covariance of $\mathbf{x}_{k}$ after $\mathbf{y}_{k}$ is observed. At the time $k$, the prediction step of the Kalman filter is:
\begin{equation}\label{eq:kf_predict}
   \begin{aligned}
       & \mathbf{\hat{x}}_{k}^- = F_{k-1}\mathbf{\hat{x}}_{k-1}^+ + G_{k-1}u_{k-1}, \\
       & \hat{P}_{k}^- = F_{k-1}\hat{P}_{k-1}^+F_{k-1}^\top + \hat{Q}_{k-1},
   \end{aligned}
\end{equation}
where $\hat{Q}_{k-1}$ is the predictor's prior knowledge of the process noises.
The update step of the Kalman filter is:
\begin{equation}\label{eq:kf_update}
   \begin{aligned}
       & K_{k} = \hat{P}_{k}^-H_{k}^\top(H_{k}\hat{P}_{k}^-H_{k}^\top + \hat{R}_{k})^{-1},                      \\
       & \mathbf{\hat{x}}_{k}^+ = \mathbf{\hat{x}}_{k}^- + K_{k}(\mathbf{y}_{k} - H_{k}\mathbf{\hat{x}}_{k}^-), \\
       & \hat{P}_{k}^+ = (I-K_{k}H_{k})\hat{P}_{k}^-,
   \end{aligned}
\end{equation}
where $\hat{R}_k$ is the predictor's prior knowledge of the observation noises.
If the prior knowledge on $\hat{Q}_k$ and $\hat{R}_k$ is accurate, we denote the estimations as $\mathbf{x}^-_k, \mathbf{x}^+_k$ and covariances as $P_k^-, P_k^+$ without the hat superscript.

When the noises are assumed to be Gaussian, the conditional distributions can be explicitly evaluated.
\begin{lemma}\label{lem:kf_pdf}
   If the linear system noises are Gaussian and the information set is $\mathcal{I}_2$, there is
   \begin{equation}\label{eq:kf_pdf}
      p_{\mathbf{y}_k\mid \mathbf{y}_{1:k-1}}(\cdot \mid y_{1:k-1}) \sim \mathcal{N}(H_k{x}^-_k, H_k{P}^-_kH_k^\top + {R}_k).
   \end{equation}
\end{lemma}
\begin{proof}
   Please see Appendix \ref{app:lem:kf_pdf}.
\end{proof}

\subsection{KF Moment-based Robust Probabilistic Predictor}
When the noises are not only known but also subjected to Gaussian distributions, it is shown that the Kalman filter can be viewed as an online predictor. When the noises are unknown, it still provides estimations for the first two moments of the conditional distributions, which can be leveraged in step $2$ of designing a moment robust probabilistic predictor for the linear system \eqref{eq:linear-sds}. A complete procedure is provided in Algorithm \ref{alg:general-kf}.

\begin{algorithm}[h]
   \renewcommand{\algorithmicrequire}{\textbf{Input:}}
   \renewcommand{\algorithmicensure}{\textbf{Output:}}
   \caption{KF Moment-based Robust Probabilistic Predictor}
   \label{alg:general-kf}
   \begin{algorithmic}[1]
      \REQUIRE $\hat{x}^+_0, \hat{P}^+_0, n, \hat{Q}_k, \hat{R}_k, \hat{u}_{0:n-1}, \operatorname{INF}$
      \STATE Initialization: $k \leftarrow 0, m \leftarrow 3$
      \REPEAT
      \STATE $m \leftarrow m-1$
      \REPEAT
      \STATE $k \leftarrow k + 1$
      \STATE Update $ \hat{x}^-_k, \hat{P}^-_k $ based on Equation \eqref{eq:kf_predict}
      \STATE Design a $m$-th moment robust distribution $\hat{p}_{\mathbf{y}_k\mid \mathbf{y}_{1:k-1}}$
      \STATE Observe $y_k$
      \STATE Update $\mathcal{L}(\mathscr{F}, y_{1:k})$ by \[\mathcal{L}(\mathscr{F}, y_{1:k-1}) + \log \hat{p}_{\mathbf{y}_k \mid \mathbf{y}_{1:k-1}}(y_k\mid y_{1:k-1})\]
      \STATE Update $ \hat{x}^+_k, \hat{P}^+_k $ based on Equation \eqref{eq:kf_update}
      \UNTIL $k=n$
      \UNTIL $\mathcal{L}(\mathscr{F}, y_{1:n}) < \operatorname{INF}$
      \ENSURE $\{p_{\mathbf{y}_k \mid \mathbf{y}_{1:k-1}}(\cdot \mid y_{1:k-1})\}_{k=1}^n$, $\{\mathcal{L}(\mathscr{F}, y_{1:k})\}_{k=1}^n$
   \end{algorithmic}
\end{algorithm}

When the distributions are Gaussian, the log-likelihood of the KF second-order moment predictor can be explicitly calculated in the next theorem.
\begin{theorem}\label{thm:evl-gauss}
   When the noises and initial state are subjected to unknown Gaussian distributions, the expected log-likelihood functional of the KF-based predictor is
   \begin{equation*}\label{eq:evl-gauss}
      \begin{aligned}
         \mathcal{L}(\mathscr{F}, \mathbf{y}_{1:n})= -\frac{1}{2}\sum_{k=1}^{n} & \left\{e_k^\top\hat{\Sigma}_k^{-1}e_k+\tr[\hat{\Sigma}_k^{-1}\Sigma_k]+\right. \\
                                                                                & \ln|\hat{\Sigma}_kR_kQ_{k-1}P_{k-1}^+|\!-\!\ln|\Sigma_k|                       \\
                                                                                & +\left.(2\ln(2\pi)+2)d_x + \ln(2\pi)d_y\right\},
      \end{aligned}
   \end{equation*}
   where
   \begin{equation*}
      \left\{\begin{aligned}
          & \Sigma_k = H_k{P}_k^{-}H_k^\top + {R}_k                               \\
          & \hat{\Sigma}_k = H_k\hat{P}_k^{-}H_k^\top + \hat{R}_k                 \\
          & e_k = H_kF_0(x_0 - \hat{x}_0^+)\prod_{i=1}^{k-1}F_{i}(I-K_{i}H_{i}) .
      \end{aligned}\right.
   \end{equation*}
\end{theorem}
\begin{proof}
   Please see Appendix \ref{app:thm:evl-gauss}.
\end{proof}
\begin{remark}
   Unlike the covariance matrix $P_{\infty}^{-}$, expression \eqref{eq:evl-gauss} provides a more concrete characterization of the prediction performance of KF. Moreover, it is clearer how different factors are quantitatively combined to influence the prediction performance.
\end{remark}

\begin{corollary}
   When the noises and initial state are subjected to known Gaussian distributions, the expected log-likelihood functional is given by
   \begin{equation}
      \begin{aligned}
         \mathcal{L}(\mathscr{F}, \mathbf{y}_{1:n}) \!=\!-\frac{1}{2}\sum_{k=1}^{n} & \{  \ln(|R_kQ_kP_{k-1}^+|) + 2(\ln(2\pi)+1)d_x \\
                                                                                    & + (\ln(2\pi)+1)d_y \}.
      \end{aligned}
   \end{equation}
\end{corollary}

However, if the normality assumption is utilized when the noises are non-Gaussian, this prediction performance may degenerate a lot. Consider a toy example as follows, which points out the danger of inappropriate Gaussian assumptions in a noise-unknown setting.

\begin{example}
   Suppose $\mathbf{y}_1$ is subjected to a two-point discrete distribution that $\pr(\mathbf{y}_1 = -1) = \pr(\mathbf{y}_1 = 1) = 0.5$, it follows imediately that $\E \mathbf{y}_1 = 0$ and $\cov \mathbf{y}_1 = 1$. A Kalman filter, even though knowing the first two moments of $\mathbf{y}_1$, uses the Gaussian distribution to make predictions, i.e., $\hat{p}_{\mathbf{y}_1} \sim \mathcal{N}(0, 1)$. Then, the log-likelihood functional is $\mathcal{L}(\hat{p}_{\mathbf{y}_1}, y_1)= -\frac{1}{2}\log(2\pi)-\frac{1}{2}$. If $\mathbf{y}_1$ is also subjected to $\mathcal{N}(0, 1)$, there is $\mathcal{L}(\hat{p}_{\mathbf{y}_1}, y_1) \geq \frac{1}{2}\log(2\pi)-\frac{1}{2}$ for any $y_1\in[-1,1]$, which accounts for nearly 70 percent of all the possible $y_1$.
\end{example}

This toy example reveals that covariance is insufficient to characterize the prediction performance when the noises are unknown, while the log-likelihood functional can distinguish the effect caused by inappropriate Gaussian assumptions in Kalman filters.

\subsection{Discussion and Generalization}
\begin{figure*}[t]
   \centering
   \subfigure[Second Moment Predictor]{
      \includegraphics[width = 0.23\textwidth]{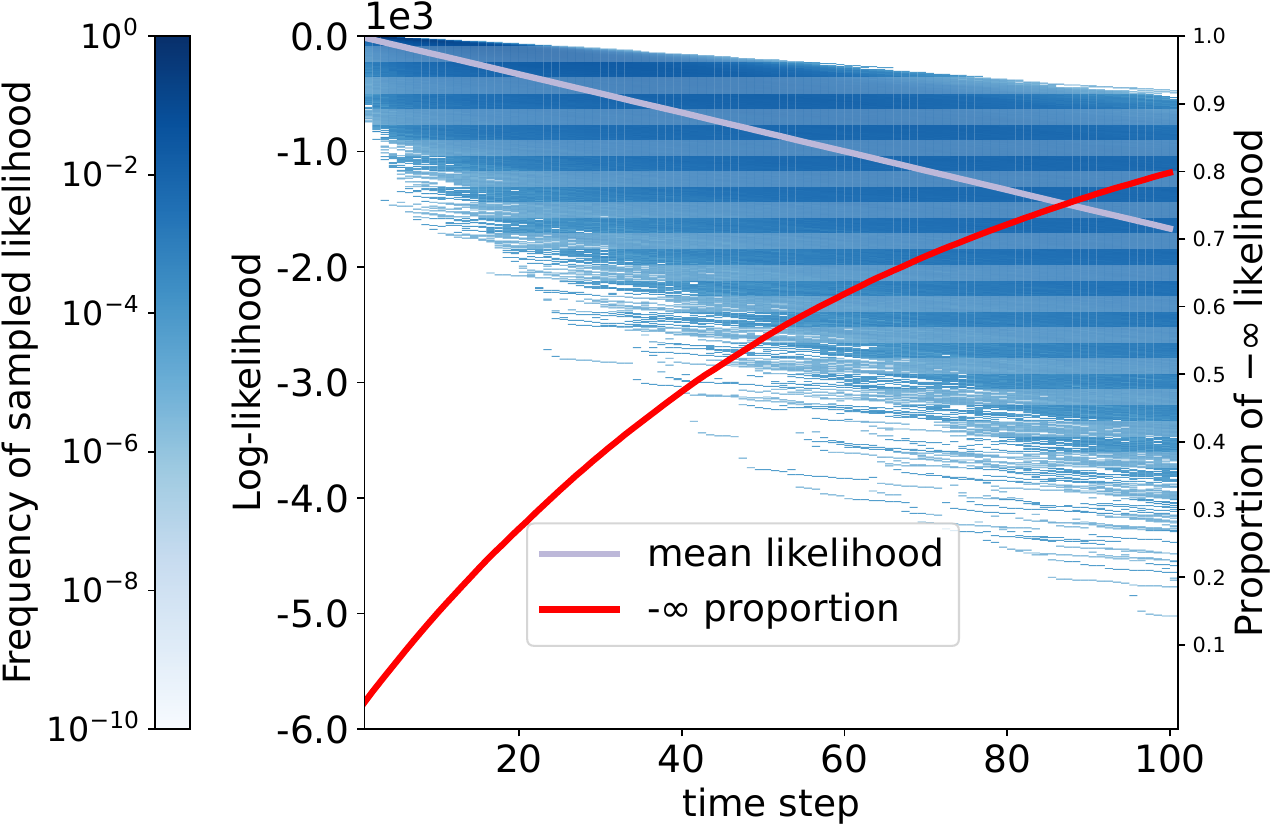}
      \label{sim:dy1:guassian}
   }
   \subfigure[First Moment Predictor]{
      \includegraphics[width = 0.23\textwidth]{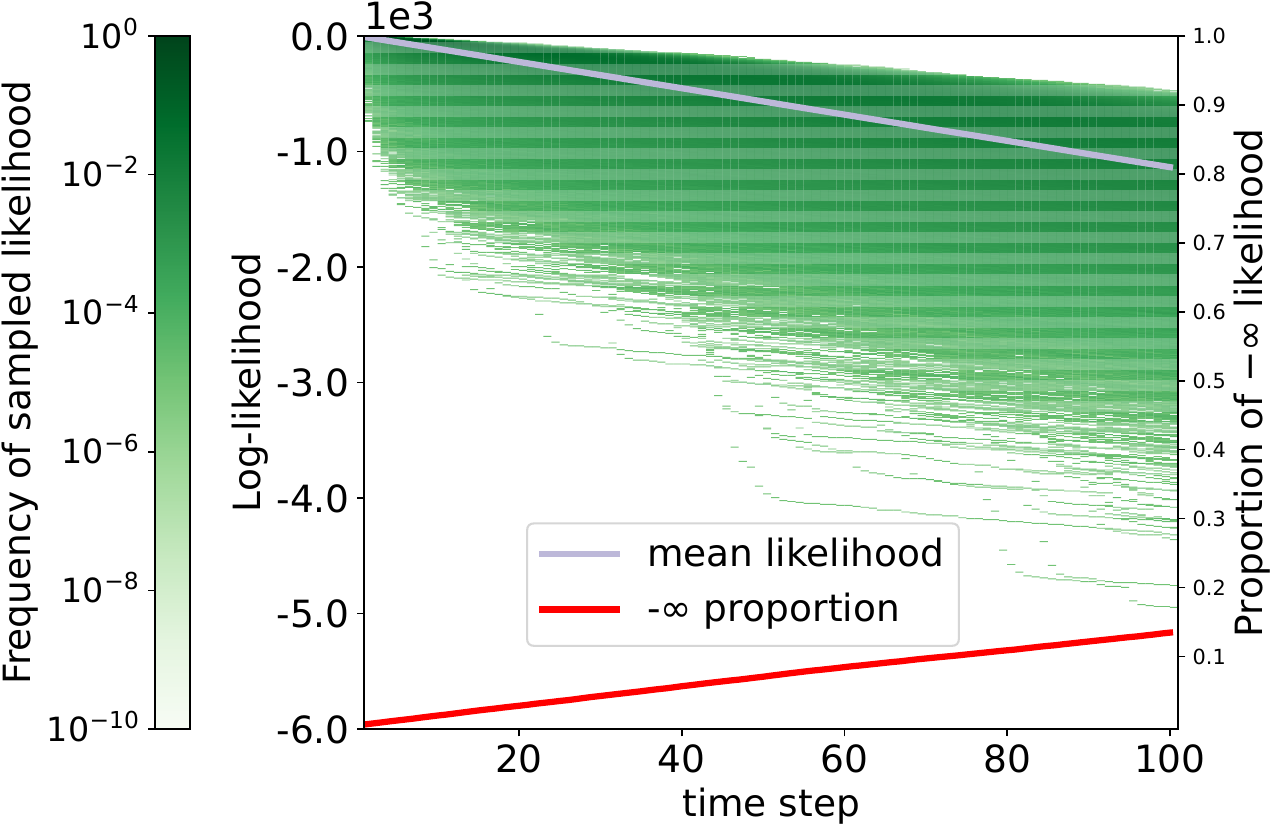}
      \label{sim:dy1:laplace}
   }
   \subfigure[$t(2)$ Predictor]{
      \includegraphics[width = 0.23\textwidth]{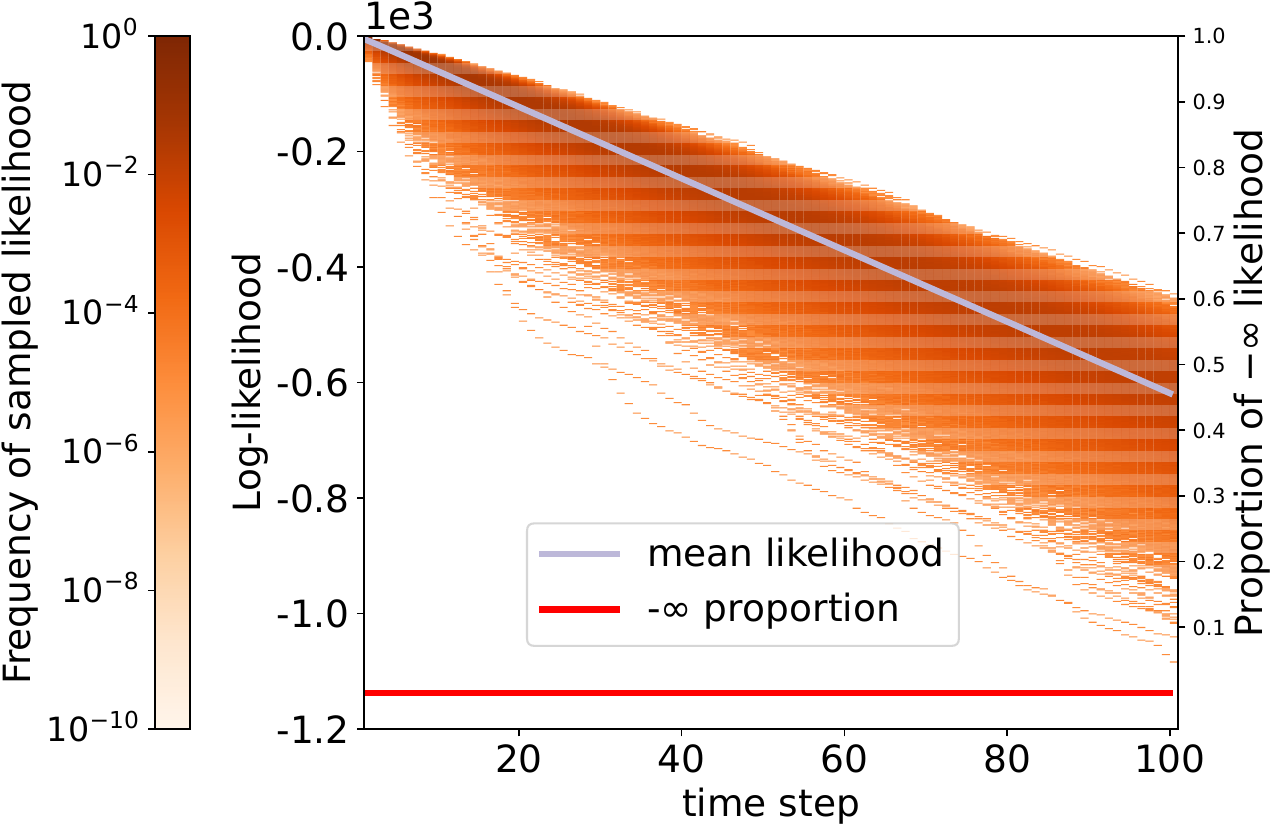}
      \label{sim:dy1:t2}
   }
   \subfigure[$t(1)$ Predictor]{
      \includegraphics[width = 0.23\textwidth]{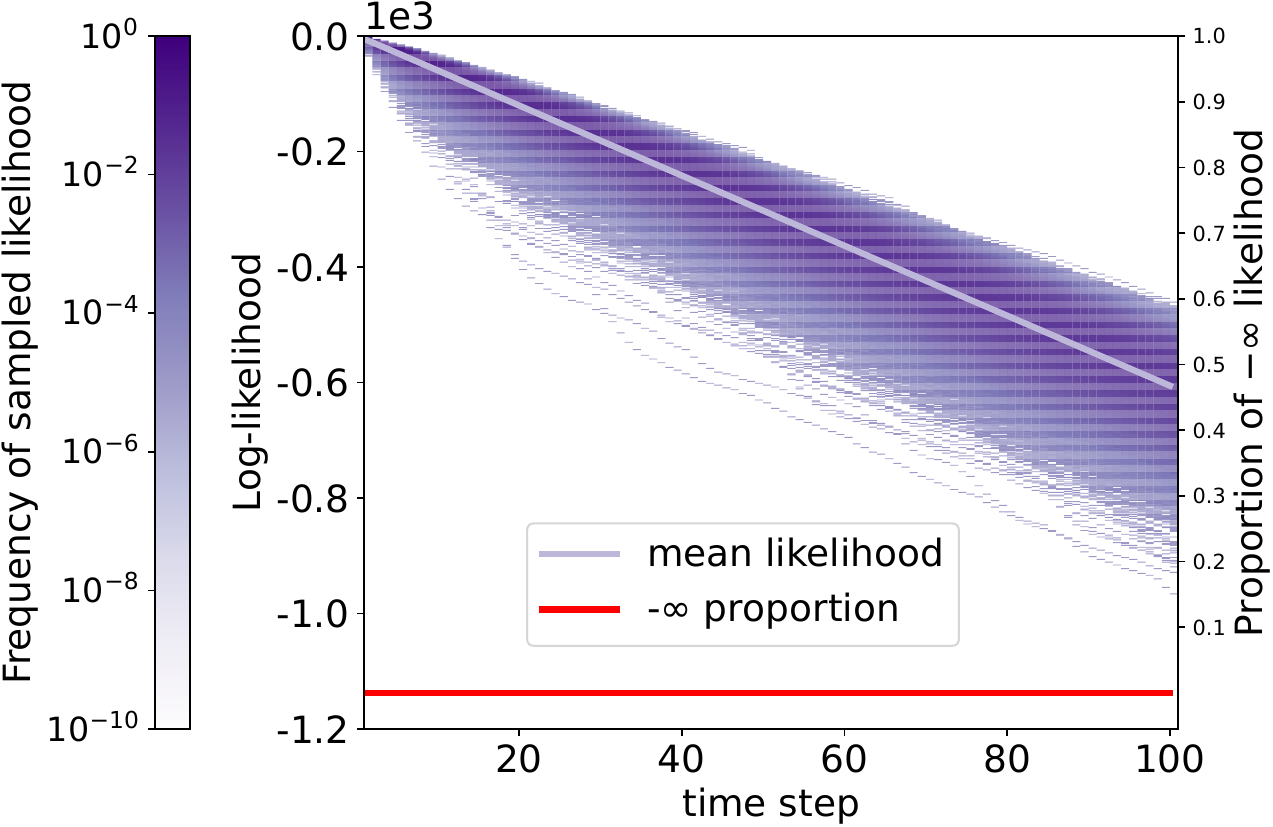}
      \label{sim:dy1:t1}
   }
   \label{sim:dy1}
   \caption{Performance of different predictors on trajectories of SDS $\Phi_1$ with $t(1)$ process noises.}
\end{figure*}

\begin{figure*}[t]
   \centering
   \subfigure[Second Moment Predictor]{
      \includegraphics[width = 0.23\textwidth]{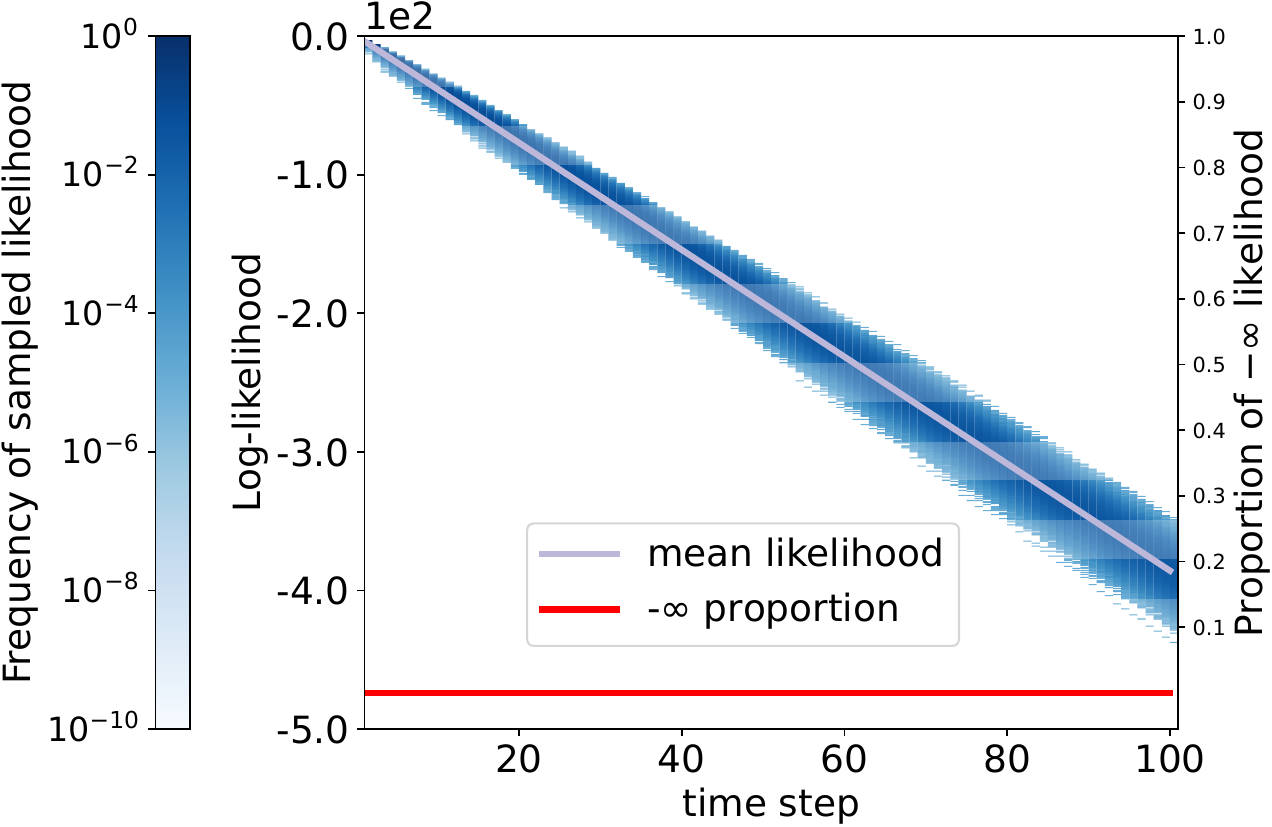}
      \label{sim:dy2:guassian}
   }
   \subfigure[First Moment Predictor]{
      \includegraphics[width = 0.23\textwidth]{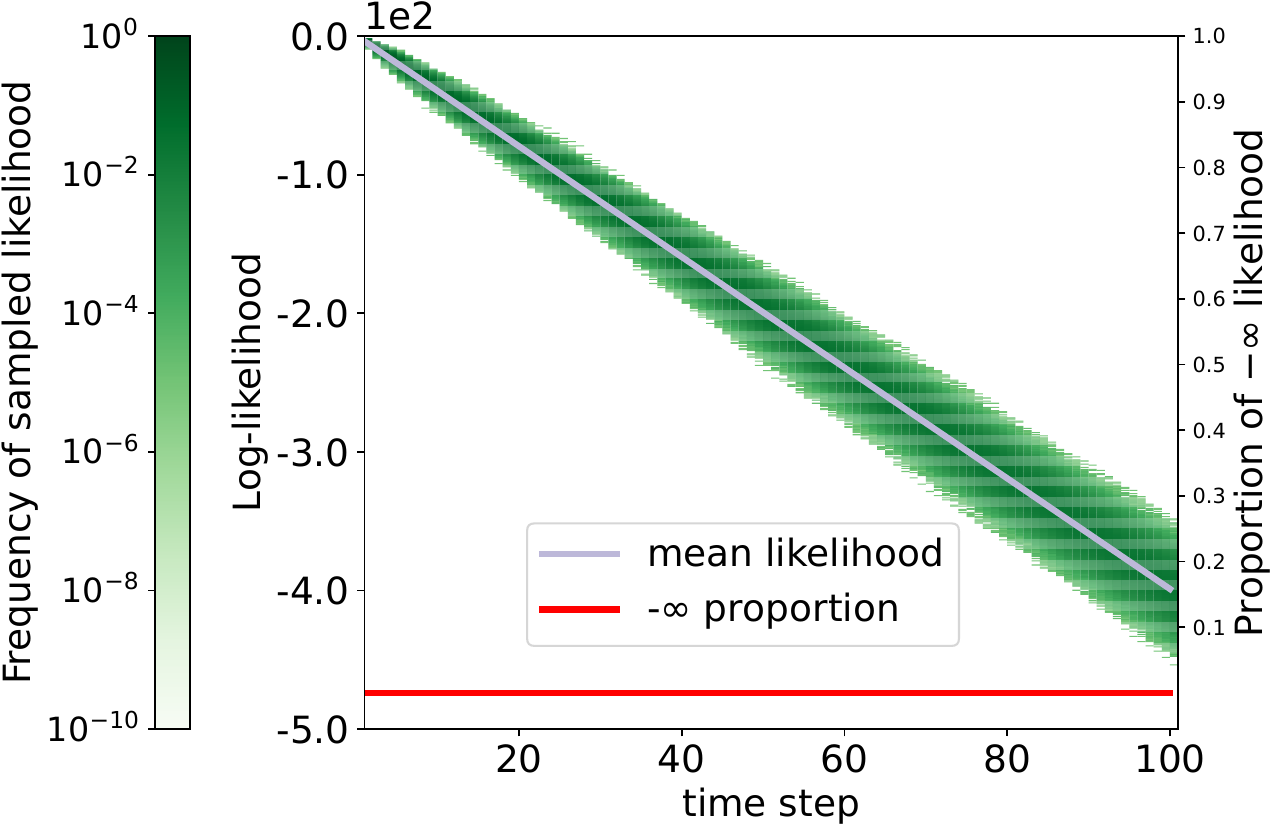}
      \label{sim:dy2:laplace}
   }
   \subfigure[$t(2)$ Predictor]{
      \includegraphics[width = 0.23\textwidth]{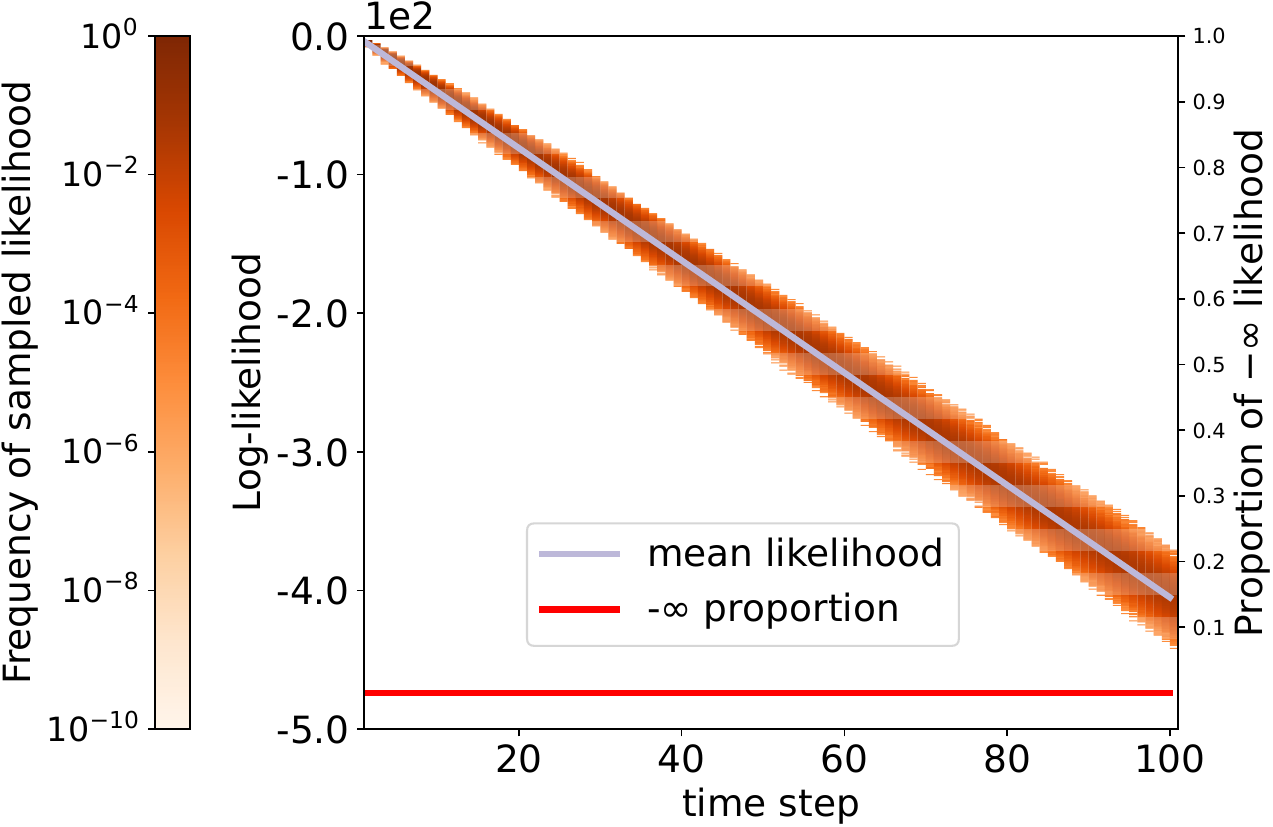}
      \label{sim:dy2:t2}
   }
   \subfigure[$t(1)$ Predictor]{
      \includegraphics[width = 0.23\textwidth]{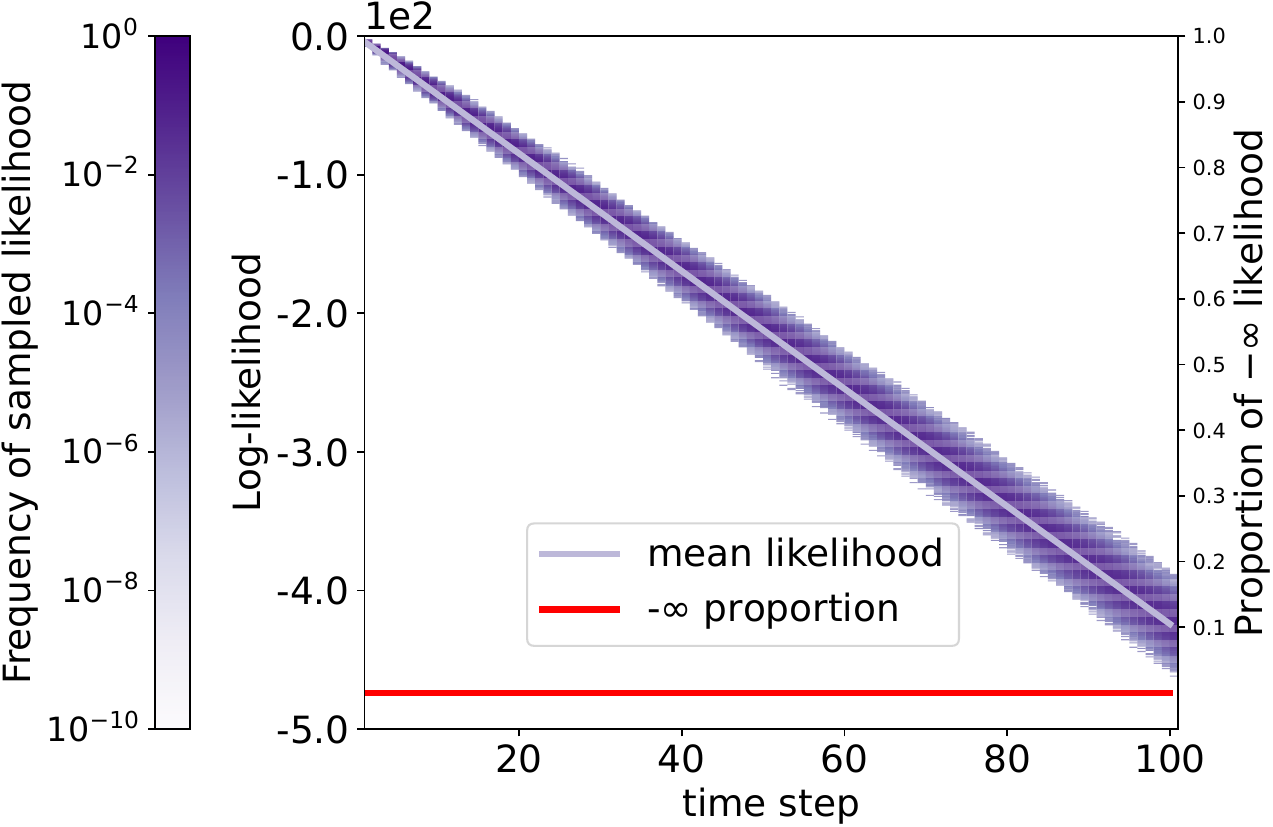}
      \label{sim:dy2:t1}
   }
   \caption{Performance of different predictors on trajectories of SDS $\Phi_2$ with Gaussian process noises.}
   \label{sim:dy2}
\end{figure*}
The classical Kalman filter uses a linear combination of $\mathbf{y}_{1:k-1}$ to approximate $\mu_1(\mathbf{y}_{1:k-1})$. To make sure that this approximation is good enough, the Kalman filter utilizes the property that conditional expectation has the minimum variance among all the functions, i.e., \[\E(\mathbf{y}_k\mid \mathbf{y}_{1:k-1}) = \arg\min_{l(\mathbf{y}_{1:k-1})} \E\|\mathbf{y}_k - l(\mathbf{y}_{1:k-1})\|_2^2.\] Then, when $l$ is limited to the linear function space, the optimal linear estimator can be explicitly attained based on the knowledge of covariances. This method can be easily generalized to higher-order moments' estimation.
\begin{theorem}
   Given an $m$-th order information set $\mathcal{I}_m$, if any $r$-th moment of the conditional distribution $p_{\mathbf{y}_k\mid \mathbf{y}_{1:k-1}}$ can be determined by a function of $\mathbf{y}_{1:k-1}$, i.e.,
   \[\E(\mathbf{y}_k^{\beta} \mid \mathbf{y}_{1:k-1}) = \mu_{\beta}(\mathbf{y}_{1:k-1}),\]
   where $\beta$ is any multi-index with $|\beta| = r$, there is \[m \geq r+1.\]
\end{theorem}
\begin{proof}
   Suppose $|\beta| + 1 > m$, according to the property of the conditional expectation, there is
   \begin{equation*}
      \begin{aligned}
         \E(\mathbf{y}_k^{\beta} \mid \mathbf{y}_{1:k-1}) = & \arg\min_{\rho} \E \|\mathbf{y}_k^\beta - \rho(\mathbf{y}_{1:k-1})\|_2^2.
      \end{aligned}
   \end{equation*}
   Notice $\E \|\mathbf{y}_k^\beta - \rho(\mathbf{y}_{1:k-1})\|_2^2 = \E \mathbf{y}_k^{2\beta} + \E\rho(\mathbf{y}_{1:k-1})^2- 2\E(\mathbf{y}_k^\beta \rho(\mathbf{y}_{1:k-1})) $, then the value of $\E(\mathbf{y}_k^\beta \rho(\mathbf{y}_{1:k-1}))$ is not included in the information set $\mathcal{I}_m$, which means that $\rho$ cannot be uniquely determined.
\end{proof}

\section{Numerical Simulations}
In this section, we apply the KF moment-based robust probabilistic predictor to predict the output trajectories of different linear SDSs. To statistically evaluate the prediction performance of an online predictor, we randomly simulate $100,000$ trajectories with each trajectory's length fixed to be $100$, calculate the predictor's log-likelihood performance on each trajectory, and visualize how the performances on these trajectories are distributed at each time step. Then, following the three-step procedure of designing a KF moment-based robust probabilistic predictor, we test the polynomial moment-based robust probabilistic predictors with the order ranging from $2$ to $0$ and the non-polynomial moment-based robust probabilistic predictors.
\subsection{Simulation Setup \& Visualization}
Consider the following linear system $\Phi$:
\begin{equation}\label{eq:sds_sim}
   \Phi_0:
   x_{k+1} = \begin{pmatrix*}
      1 & 1\\
      0 & 1
   \end{pmatrix*}x_k + w_k,~
   y_k     = \begin{pmatrix*}
      1 & 0\\
      0 & 1
   \end{pmatrix*}x_k + v_k,
\end{equation}
with the initial state $x_0 = [1, 2]^\top$.
Based on the structure of $\Phi$, we simulate two kinds of SDS: the first one $\Phi_1$ has $t(1)$ process noises, the second one $\Phi_2$ has standard Gaussian process noises, and both of them have standard Gaussian observation noises.

To run the KF moment-based robust probabilistic predictor in Algorithm \ref{alg:general-kf}, the prior knowledge of the first two moments of $W$ and $V$ are chosen as $\E(W) = \E(V) = \boldsymbol{0}$ and
\begin{equation*}
   \cov(W) = Q = \begin{pmatrix*}
      1 & 0 \\
      0   & 1
   \end{pmatrix*},~
   \cov(V) = R = \begin{pmatrix*}
      1 & 0 \\
      0   & 1
   \end{pmatrix*}.
\end{equation*}

To visualize how the predictors act on different trajectories generated from an SDS, we want to present the likelihood-time plot for each trajectory on a single figure. However, there are two difficulties: first, since the trajectory number is large, explicitly drawing each trajectory's likelihood makes the figure both messy and short of information; second, if an unrobust predictor is used, there are trajectories that decrease the likelihood too much (beyond the numeric limitation of a 64-bit PC), which cannot be visualized by a line plot. For the first problem, we use a heat map to show how the likelihoods of different trajectories are distributed at each time step and use the mean of trajectories' likelihood to show how the expected performance changed as the time step goes from $0$ to $100$. In the heat map, each grid's color visualizes the frequency of trajectories that possess prediction performance located in the according interval. For the second problem, we add a right y-axis to denote the proportion of trajectories that have $-\infty$ likelihood and show how this rate grows as time step increases.

\subsection{Results of $\Phi_1$}
For the SDS $\Phi_1$, an appropriate information set is $\mathcal{I}_0$, since the first moment of $t(1)$ does not exist. If the information set's order is overestimated by $2$ or $1$, the accordingly designed KF moment-based predictor is no longer robust, and the prediction performance will be quite bad. As the second-order predictor in Fig. \ref{sim:dy1:guassian} shows, i) the proportion of trajectories with $-\infty$ log-likelihood continues growing to $80\%$ from time $0$ to $100$; ii) for those trajectories with finite log-likelihood, their mean decreases fast and their variance is quite large such that many outliers exist.  The first-order predictor in Fig. \ref{sim:dy1:laplace} presents similar features, but the growing speed of $-\infty$ proportion is much slower, and those trajectories with finite log-likelihood have larger means and smaller variances.

On the contrary, the $t(1)$ and $t(2)$ based predictor keep $0\%$ proportion of $-\infty$ likelihood. Moreover, their trajectories's means are significantly larger and their variances are significantly smaller, as shown in Fig. \ref{sim:dy1:t1} and Fig. \ref{sim:dy1:t2}. Of course, $t(1)$ predictor outweighs $t(2)$ predictor a little bit in the mean and variance of likelihoods, and this makes sense since the trajectories are sampled from $\Phi_1$, where the process noises are subjected to $t(1)$.

\subsection{Results of $\Phi_2$}
For the SDS $\Phi_2$, an appropriate information is $\mathcal{I}_2$, since the first two moments of the noises all exist. Therefore, using any moment-based robust probabilistic predictors with an order of no more than $2$ will be robust. As Fig. \ref{sim:dy2} shows, each predictor is robust with $0\%$ proportion of $-\infty$ trajectories. The second-order moment robust probabilistic predictor performs best with the largest log-likelihood, and as the order decreases, the performance has a mild decrease. However, those conservative robust probabilistic predictors do not have too much performance degradation, and the $t(2)$ and $t(1)$ predictor even has smaller variances.

\section{Conclusion}
In conclusion, this paper addresses the issue of robustifying probabilistic predictions for SDSs, which inherently encompasses uncertainty quantification through distribution predictions. We introduce the concept of likelihood functional as a generalized measure of likelihood function to evaluate the performance of probabilistic predictors, which is proved to be a proper scoring rule. By leveraging this metric, we propose a comprehensive framework to assess the robustness of predictors and investigate the impact of different information sets on robustness. Our framework enables the design of robust probabilistic predictors through functional optimization problems tailored to specific information sets. Notably, we develop a class of moment-based optimal robust probabilistic predictors and present a practical implementation algorithm utilizing the Kalman filter. Through extensive numerical simulations, we provide detailed insights and validation of our findings.

Overall, this research contributes to the development of robust probabilistic predictors for SDSs, advancing our understanding and capabilities in making robust predictions with UQ. Future works include developing more learning-based online robust probabilistic predictors to further improve the prediction performance.

\begin{appendices}
   \section{Proof of Theorem \ref{thm:evl-1}}\label{app:thm:evl-1}
   To begin with, the expected log-likelihood functional can be decomposed as
   \begin{equation*}
      \begin{aligned}
         \mathcal{L}(\mathscr{F}, \mathbf{y}_{1:n}) \overset{(\mathrm{i})}{=} & \E_{y_{1:n}}\log \hat{p}_{\mathbf{y}_{1:n}}(y_{1:n})                                                \\
         \overset{(\mathrm{ii})}{=}                                           & \E_{y_{1:n}} \sum_{k=1}^{n}\log \hat{p}_{\mathbf{y}_k \mid \mathbf{y}_{1:k-1}}(y_k\mid y_{1:k-1})   \\
         \overset{(\mathrm{iii})}{=}                                          & \sum_{k=1}^{n} \E_{y_{1:n}} \log \hat{p}_{\mathbf{y}_k \mid \mathbf{y}_{1:k-1}}(y_k\mid y_{1:k-1})  \\
         \overset{(\mathrm{iv})}{=}                                           & \sum_{k=1}^{n} \E_{y_{1:k}} \log \hat{p}_{\mathbf{y}_k \mid \mathbf{y}_{1:k-1}}(y_k\mid y_{1:k-1}),
      \end{aligned}
   \end{equation*}
   where $\mathrm{(i)}$ follows from the definition of $\mathcal{L}(\mathscr{F}, \mathbf{y}_{1:n})$, $\mathrm{(ii)}$ follows from the chain rule equation \eqref{eq:chain}, $\mathrm{(iii)}$ exchange the expectation with the finite summation, and $\mathrm{(iv)}$ holds because  $\log \hat{p}_{\mathbf{y}_k \mid \mathbf{y}_{1:k-1}}(y_k\mid y_{1:k-1})$ only depends on $y_{1:k}$. Furthermore, we have
   \begin{equation*}
      \begin{aligned}
                                     & \E_{y_{1:k}} \log \hat{p}_{\mathbf{y}_k \mid \mathbf{y}_{1:k-1}}(y_k\mid y_{1:k-1})                                                      \\
         \overset{\mathrm{(i)}}{=}   & \E_{y_{1:k-1}}\!\left\{\E_{y_k}\left[\log \hat{p}_{\mathbf{y}_k \mid \mathbf{y}_{1:k-1}}(y_k\mid y_{1:k-1})\mid y_{1:k-1}\right]\right\} \\
         \overset{\mathrm{(ii)}}{=}  & \E_{y_{1:k-1}}\!\left\{-\mathrm{H}(p_{\mathbf{y}_k\mid \mathbf{y}_{1:k-1}})\!-\! D_{KL}\!\left(q_k\Vert \hat{q}_k\right) \right\}        \\
         \overset{\mathrm{(iii)}}{=} & \!\!-\!\mathrm{H}(\mathbf{y}_k\mid \mathbf{y}_{1:k-1}) \!-\! \E_{y_{1:k-1}}D_{KL}\!\left(q_k\Vert \hat{q}_k\right),
      \end{aligned}
   \end{equation*}
   where $\mathrm{(i)}$ follows from the tower property of conditional expectation, $\mathrm{(ii)}$ follows from the definition of cross entropy and its decomposition into the sum of differential entropy and KL-divergence, $\mathrm{(iii)}$ holds according to the definition conditional entropy. Finally, according to the chain rule that decomposes the joint entropy into conditional entropies, 
   \begin{equation*}
      \begin{aligned}
           & \sum_{k=1}^{n} \E_{y_{1:k}} \log \hat{p}_{\mathbf{y}_k \mid \mathbf{y}_{1:k-1}}(y_k\mid y_{1:k-1})                           \\
         = & \sum_{k=1}^{n} \!-\mathrm{H}(\mathbf{y}_k\mid \mathbf{y}_{1:k-1}) \!-\!\E_{y_{1:k-1}}D_{KL}\!\left(q_k\Vert \hat{q}_k\right) \\
         = & \!-\mathrm{H}(\mathbf{y}_{1:n}) \!-\! \sum_{k=1}^{n}\E_{y_{1:k-1}} D_{KL}\!\left(q_k\Vert \hat{q}_k\right),
      \end{aligned}
   \end{equation*}
   and the proof is completed.

   \section{Proof of Theorem \ref{lem:finite-condition}}\label{app:lem:finite-condition}
   The expected log-likelihood can be decomposed as the sum of one-step expected log-likelihood as follows,
   \begin{equation*}
      \begin{aligned}
           & \mathcal{L}(\mathcal{F}, \mathbf{y}_{1:n})
         = \E_{y_{1:n}}\sum_{k=1}^{n} \log \hat{p}_{\mathbf{y}_k \mid \mathbf{y}_{1:k-1}}(y_k \mid y_{1:k-1})                                                                                                        \\
         = & \sum_{k=1}^{n}\E_{y_{1:k}} \log \hat{p}_{\mathbf{y}_k \mid \mathbf{y}_{1:k-1}}(y_k \mid y_{1:k-1})                                                                                                      \\
         = & \sum_{k=1}^{n}\E_{y_{1:k\!-\!1}}\!\!\int\!\! p_{\mathbf{y}_k \mid \mathbf{y}_{1:k\!-\!1}}(s\!\mid\! y_{1:k-1}) \log \hat{p}_{\mathbf{y}_k \mid \mathbf{y}_{1:k\!-\!1}}(s \!\mid\! y_{1:k\!-\!1}) \dd s.
      \end{aligned}
   \end{equation*}
   Additionally, the effect of $\mathbf{u}_k$ is to control the conditional distribution $p_{\mathbf{y}_{k+1} \mid \mathbf{y}_{1:k}}$ with the constraint that the $m$-th order moments of $\mathbf{u}_k$ exist.

   On the one hand, if there is \[\min\limits_{p}\int\!\! p_{\mathbf{y}_k \mid \mathbf{y}_{1:k\!-\!1}}(s\!\mid\! y_{1:k-1}) \log \hat{p}_{\mathbf{y}_k \mid \mathbf{y}_{1:k\!-\!1}}(s \!\mid\! y_{1:k\!-\!1}) \dd s > -\infty\] holding for each trajectory $y_{1:k-1}$, then the expectation over $\mathbf{y}_{1:k-1}$ is guaranteed to be lower bounded, and the sum over $k$ from $1$ to $n$ is also lower bounded.

   On the other hand, if there exists a trajectory set $A$ with $\pr(\mathbf{y}_{1:k-1} \in A) > 0$ such that $\forall y_{1:k-1} \in A$ \[\min\limits_{p}\int\!\! p_{\mathbf{y}_k \mid \mathbf{y}_{1:k\!-\!1}}(s\!\mid\! y_{1:k-1}) \log \hat{p}_{\mathbf{y}_k \mid \mathbf{y}_{1:k\!-\!1}}(s \!\mid\! y_{1:k-1}) \dd s = -\infty,\] then we have
   \begin{equation*}
      \begin{aligned}
          & \min\limits_{p} \E_{y_{1:k-1}}\int\!\! p_{\mathbf{y}_k \mid \mathbf{y}_{1:k\!-\!1}}(s\!\mid\! y_{1:k-1}) \log \hat{p}_{\mathbf{y}_k \mid \mathbf{y}_{1:k\!-\!1}}(s \!\mid\! y_{1:k-1}) \dd s \\
          & = -\infty,
      \end{aligned}
   \end{equation*}
   and the proof is completed.

   \section{Proof of Theorem \ref{thm:robust-form}}\label{app:thm:robust-form}
   We form an augmented variational problem for \eqref{eq:maxmin-aux} by defining
   \begin{equation*}
      \begin{aligned}
         J[p, \lambda] = & \int p_{\mathbf{y}_k \mid \mathbf{y}_{1:k-1}}(s \mid y_{1:k-1}) \log \hat{p}_{\mathbf{y}_k \mid \mathbf{y}_{1:k-1}}(s \mid y_{1:k-1}) \dd s                \\
                         & -\sum_{i=0}^m\sum_{|\alpha|=i} \lambda_\alpha \left[\int s^\alpha p_{\mathbf{y}_k \mid \mathbf{y}_{1:k-1}}(s\!\mid\! y_{1:k-1})\dd s - \mu_\alpha \right],
      \end{aligned}
   \end{equation*}
   where $\lambda_\alpha\in \R$ are the constant Lagrange multipliers corresponding to the integral constraints.

   Doing variation on $J[p, \lambda]$ with respect to $p_{\mathbf{y}_k \mid \mathbf{y}_{1:k-1}}$, the Euler-Lagrange equation admits a necessary condition that the optimal value is bounded below:
   \begin{equation*}
      \begin{aligned}
         \log p_{\mathbf{y}_k \mid \mathbf{y}_{1:k-1}}(\cdot \mid y_{1:k-1}) - \sum_{i=0}^{m} \sum_{|\alpha|=i} \lambda_\alpha s^\alpha = 0.
      \end{aligned}
   \end{equation*}

   Substituting this equation to the optimization objective, it follows that
   \begin{equation*}
      \begin{aligned}
           & \int p_{\mathbf{y}_k \!\mid \mathbf{y}_{1:k-1}}(s\!\mid\! y_{1:k-1})\log \hat{p}_{\mathbf{y}_k \mid \mathbf{y}_{1:k-1}}(s\!\mid\! y_{1:k-1}) \dd s \\
         = & \int p_{\mathbf{y}_k \!\mid \mathbf{y}_{1:k-1}}(s\!\mid\! y_{1:k-1})\sum_{i=0}^{m} \sum_{|\alpha|=i} \lambda_\alpha s^\alpha \dd s                 \\
         = & \sum_{i=0}^{m} \sum_{|\alpha|=i}\int p_{\mathbf{y}_k \!\mid \mathbf{y}_{1:k-1}}(s\!\mid\! y_{1:k-1}) \lambda_\alpha s^\alpha \dd s                 \\
         = & \sum_{i=0}^{m} \sum_{|\alpha|=i} \lambda_\alpha \mu_\alpha(\mathbf{y}_k \mid y_{1:k-1}).                                                           \\
      \end{aligned}
   \end{equation*}
   Since $|\mu_\alpha(\mathbf{y}_k \mid y_{1:k-1})| < \infty$ for $y_{1:k-1}\in\R^{(k-1)d_y}$ a.s., we have verified that the optimal value to \eqref{eq:maxmin-aux} is indeed bounded below, and the proof is completed.

   \section{Proof of Theorem \ref{thm:maxmin-2}}\label{app:thm:maxmin-2}
   Considering Theorem \ref{thm:robust-form} and the information that $\operatorname{supp}(\mathbf{y}_k)$ is both upper and lower unbounded, we have the optimal second-moment robust probabilistic predictor $\hat{p}_{\mathbf{y}_k \mid \mathbf{y}_{1:k-1}}^\star$ should be a multivariate Gaussian distribution:
   \[ (2 \pi)^{-d_y / 2} \operatorname{det}(\hat{\Sigma})^{-1 / 2} e^{\left(-\frac{1}{2}(s-\hat{z})^{\top} \hat{\Sigma}^{-1}(s-\hat{z})\right)},\]
   and there is
   \begin{equation*}
      \begin{aligned}
                                       & \int\! p_{\mathbf{y}_k \!\mid \mathbf{y}_{1:k-1}}(s\!\mid\! y_{1:k-1})\log \hat{p}_{\mathbf{y}_k \mid \mathbf{y}_{1:k-1}}(s\!\mid\! y_{1:k-1}) \dd s                    \\
         =                             & -\frac{d_y}{2}\log(2\pi) - \frac{1}{2}\log\operatorname{det}(\hat{\Sigma})                                                                                              \\
                                       & + \int\! p_{\mathbf{y}_k \!\mid \mathbf{y}_{1:k-1}}(s\!\mid\! y_{1:k-1}) \left(-\frac{1}{2}(s-\hat{z})^{\top} \hat{\Sigma}^{-1}(s-\hat{z})\right)\dd s                          \\
         =                             & -\frac{d_y}{2}\log(2\pi) - \frac{1}{2}\log\operatorname{det}(\hat{\Sigma})                                                                                              \\
                                       & + \int\! p_{\mathbf{y}_k \!\mid \mathbf{y}_{1:k-1}}(s\!\mid\! y_{1:k-1}) \left(-\frac{1}{2}(s-z + z - \hat{z})^{\top}\right.                                                \\
                                       & \qquad \left.\hat{\Sigma}^{-1}(s-z + z-\hat{z})\right)\dd s                                                                                                                       \\
         \overset{\mathrm{(i)}}{=}     & \!-\!\frac{d_y}{2}\log(2\pi) \!-\! \frac{1}{2}\log\operatorname{det}(\hat{\Sigma}) \!-\! \frac{1}{2}\langle \Sigma,\hat{\Sigma}^{-1}\rangle \!-\! \|z \!-\! \hat{z}\|_{\hat{\Sigma}^{-1}}^2 \\
         \overset{\mathrm{(ii)}}{\leq} & \!-\!\frac{d_y}{2}\log(2\pi) \!-\! \frac{1}{2}\log\operatorname{det}(\hat{\Sigma}) \!-\! \frac{1}{2}\langle \Sigma,\hat{\Sigma}^{-1}\rangle
      \end{aligned}
   \end{equation*}
   where $\mathrm{(i)}$ holds according to the definition of the inner product of matrices, and the equation of $\mathrm{(ii)}$ holds if and only if $\hat{z} = z$.

   Now we can do differential on the objective function with respect to $\hat{\Sigma}^{-1}$, such that
   \begin{equation*}
      \begin{aligned}
           & \frac{\dd }{\dd \hat{\Sigma}^{-1}} \left\{ -\frac{d_y}{2}\log(2\pi) - \frac{1}{2}\log\operatorname{det}(\hat{\Sigma}) \!-\! \frac{1}{2}\langle \Sigma,\hat{\Sigma}^{-1}\rangle \right\} \\
         = & \frac{1}{2} (\hat{\Sigma} - \Sigma).
      \end{aligned}
   \end{equation*}
   It follows that the objective function is maximized when $\hat{\Sigma} = \Sigma$. Substituting $\hat{\Sigma}$ and $\hat{z}$ by the above calculations, the optimal conditional distribution is 
   \begin{align*}
      \hat{p}_{\mathbf{y}_{k} \mid \mathbf{y}_{1:k-1}}^\star(s \mid y_{1:k-1})= & (2 \pi)^{-{d_y\over 2}} \operatorname{det}(\Sigma)^{-{1\over 2}}   \\
                                                                                & \times e^{\left(-\frac{1}{2}(s-z)^{\top} \Sigma^{-1}(s-z)\right)}.
   \end{align*}

   \section{Proof of Theorem \ref{thm:maxmin-1}}\label{app:thm:maxmin-1}
   Solving the optimal first-moment robust probabilistic predictor is equivalent to doing the following optimization problem for each $i \in \{1, \ldots, d_y\}$:
   \begin{align}
                  & \max\limits_{\lambda_0, \lambda_1} \lambda_1^{(i)} z^{(i)} + \lambda_0^{(i)} \nonumber \\
      \text{s.t.} &
      \begin{aligned}
          & \int_{\underline{y}_k^{(i)}}^{\bar{y}_k^{(i)}} e^{\lambda_1^{(i)} s^{(i)} + \lambda_0^{(i)} } \dd s^{(i)} = 1.
      \end{aligned}
   \end{align}

   i) When $\underline{y}_k^{(i)} > -\infty$ and $\bar{y}_k^{(i)} = \infty$, there is
   \begin{equation*}
      \begin{aligned}
                     & e^{\lambda_0^{(i)}}(0 - e^{\lambda_1^{(i)}\underline{y}_k^{(i)}}) = \lambda_1^{(i)} \\
         \Rightarrow & \lambda_0^{(i)} = \log(-\lambda_1^{(i)}) - \lambda_1^{(i)}\underline{y}_k^{(i)}.
      \end{aligned}
   \end{equation*}
   Substituting this into the objective function, we have
   \begin{equation*}
      \begin{aligned}
           & \lambda_1^{(i)} z^{(i)} + \lambda_0^{(i)}                                  \\
         = & \lambda_1^{(i)}(z^{(i)} - \underline{y}_k^{(i)}) + \log(-\lambda_1^{(i)}),
      \end{aligned}
   \end{equation*}
   which is maximized when $\lambda_1^{(i)} = (\underline{y}_k^{(i)} - z^{(i)})^{-1}$. Then $\lambda_0^{(i)} = -\log\left(z^{(i)} - \underline{y}_k^{(i)}\right) + \underline{y}_k^{(i)}(z^{(i)} - \underline{y}_k^{(i)})^{-1}$.

   ii) When $\underline{y}_k^{(i)} = -\infty$ and $\bar{y}_k^{(i)} < \infty$, there is
   \begin{equation*}
      \begin{aligned}
                     & e^{\lambda_0^{(i)}}(e^{\lambda_1^{(i)}\bar{y}_k^{(i)}} - 0) = \lambda_1^{(i)} \\
         \Rightarrow & \lambda_0^{(i)} = \log(\lambda_1^{(i)}) - \lambda_1^{(i)}\bar{y}_k^{(i)}.
      \end{aligned}
   \end{equation*}
   Substituting this into the objective function, we have
   \begin{equation*}
      \begin{aligned}
           & \lambda_1^{(i)} z^{(i)} + \lambda_0^{(i)}                                  \\
         = & \lambda_1^{(i)}(z^{(i)} - \underline{y}_k^{(i)}) + \log(-\lambda_1^{(i)}),
      \end{aligned}
   \end{equation*}
   which is maximized when $\lambda_1^{(i)} = (\bar{y}_k^{(i)} - z^{(i)})^{-1}$.Then $\lambda_0^{(i)} = -\log\left(\bar{y}_k^{(i)} - z^{(i)}\right) + \bar{y}_k^{(i)}(z^{(i)} - \bar{y}_k^{(i)})^{-1}$.

   ii) When $\underline{y}_k^{(i)} > -\infty$ and $\bar{y}_k^{(i)} < \infty$, the constraint is equivalent to
   \[\lambda_1^{(i)}e^{-\lambda_0^{(i)}} = e^{\bar{y}_k^{(i)}\lambda_1^{(i)}} - e^{\underline{y}_k^{(i)}\lambda_1^{(i)}}.\]
   The objective function can be replaced by $e^{ \lambda_1^{(i)} z^{(i)} + \lambda_0^{(i)}}$. Then
   \begin{equation*}
      \lambda_1^{(i)} \!=\! \arg\max_x \frac{xe^{ux}}{e^{bx} - e^{ax}} \!=\! \arg\min_x (e^{(b-u)x} - e^{(a-u)x})x^{-1},
   \end{equation*}
   where $b = \bar{y}_k^{(i)}, a = \underline{y}_k^{(i)}, u = z^{(i)}$ for the simplicity of notation. Let \[l(x) := (e^{(b-u)x} - e^{(a-u)x})x^{-1},\] then \[l^\prime(x) = \{[(b-u)x-1]e^{(b-u)x} \!-\! [(a-u)x-1]e^{(a-u)x}\}x^{-2}.\] Let \[J(x) = [(b-u)x-1]e^{(b-u)x} - [(a-u)x-1]e^{(a-u)x},\] then $J^\prime(x) = xe^{(a-u)x}[(b-u)^2e^{(b-a)x} - (a-u)^2]$. Moreover, we get that $l(x)$ is maximized at the non-zero solution to the equation $J(x) = 0$.

   \section{Proof of Lemma \ref{lem:nonpoly-moment}}\label{app:lem:nonpoly-moment}
   (i) Given a random variable $X$, the fact that its expectation exists actually yields that $\mathbb{E}\left[X^{+}\right]<\infty \text { and } \mathbb{E}\left[X^{-}\right]<\infty$, where $X^{+}=\max \{X, 0\}$ and $X^{-}=-\min \{X, 0\}$. Therefore $\E|X| = \E X^{+} + \E X^{-} < \infty$. Now if $X \sim p_{\mathbf{y}_k \mid \mathbf{y}_{1:k-1}}(\cdot \mid t)$, the information set $\mathcal{I}_1$ guarantees that $\E X < \infty$ and therefore $\E |X| < \infty$.

   (ii) When $\|s\| > 2$, there is $\log(1+\|s\|^2) < \|s\|$, then
   \begin{equation*}
      \begin{aligned}
           & \int \log(1 + s^\top s)\, p_{\mathbf{y}_k \mid \mathbf{y}_{1:k-1}}(s\mid y_{1:k-1}) \dd s               \\
         = & \int_{\|s\|\leq 2} \log(1 + s^\top s)\, p_{\mathbf{y}_k \mid \mathbf{y}_{1:k-1}}(s\mid y_{1:k-1}) \dd s \\
           & + \int_{\|s\| > 2} \log(1 + s^\top s)\, p_{\mathbf{y}_k \mid \mathbf{y}_{1:k-1}}(s\mid y_{1:k-1}) \dd s \\
         < & \log(5) + \E \|s\| < \infty,
      \end{aligned}
   \end{equation*}
   and the proof is completed.

   \section{Proof of Lemma \ref{lem:kf_pdf}} \label{app:lem:kf_pdf}
   Under the assumption of linear Gaussian system, we have the prediction error is subjected to the Gaussian distribution, and its explicit distribution is given as
   \[\mathbf{y}_k - H_k\mathbf{{x}}_k^- \sim \mathcal{N}(0, H_k{P}^-_kH_k^\top + {R}_k).\]
   Notice that when the observations $y_{1:k-1}$ and control inputs $u_{0:k-1}$ are given, the random variable $\mathbf{{x}}_k^-$ has a fixed value ${x}_k^-$, it follows that
   \begin{equation*}
      \begin{aligned}
                     & p_{\mathbf{y}_k- H_k\mathbf{{x}}_k^-\mid \mathbf{y}_{1:k-1}}(\cdot \mid y_{1:k-1}) \sim \mathcal{N}(0, H_k{P}^-_kH_k^\top + {R}_k) \\
         \Rightarrow & p_{\mathbf{y}_k\mid \mathbf{y}_{1:k-1}}(\cdot \mid y_{1:k-1}) \sim \mathcal{N}(H_k{x}^-_k, H_k{P}^-_kH_k^\top + {R}_k),
      \end{aligned}
   \end{equation*}
   and the proof is completed.

   \section{Proof of Theorem \ref{thm:evl-gauss}} \label{app:thm:evl-gauss}
   According to Lemma \ref{lem:kf_pdf}, there is \[p_{\mathbf{y}_k\mid \mathbf{y}_{1:k-1}}(\cdot \mid y_{1:k-1}) \sim \mathcal{N}(H_k{x}^-_k, \Sigma_k).\] Then, according to the definition of the KF-based predictor, there is \[\hat{p}_{\mathbf{y}_k\mid \mathbf{y}_{1:k-1}}(\cdot \mid y_{1:k-1}) \sim \mathcal{N}(H_k{\hat{x}}^-_k, \hat{\Sigma}_k).\] Next, during the iteration of the Kalman filter, we have
   \begin{equation*}
      \begin{aligned}
         \E(\mathbf{x}_{k+1} - \mathbf{\hat{x}}_{k+1}^-) = & F_k(I-K_kH_k)\E(\mathbf{x}_k - \mathbf{\hat{x}}_k^-)                    \\
         =                                                 & \left[\prod_{i=1}^{k}F_{i}(I-K_{i}H_{i})\right] F_0(x_0 - \hat{x}_0^+).
      \end{aligned}
   \end{equation*}
   It follows that
   \begin{equation*}
      \begin{aligned}
         H_kx_k^- - H_k\hat{x}_k^- & = H_k\E(\mathbf{x}_k^-) - H_k\E(\mathbf{\hat{x}}_k^-)                        \\
                                   & \overset{\mathrm{(i)}}{=} H_k\E(\mathbf{x}_k) - H_k\E(\mathbf{\hat{x}}_k^-)  \\
                                   & \overset{\mathrm{(ii)}}{=} H_k\E(\mathbf{x}_k) - (H_k\E(\mathbf{x}_k) - e_k) \\
                                   & = e_k,
      \end{aligned}
   \end{equation*}
   where $\mathrm{(i)}$ holds because $\E(\mathbf{x}_k - \mathbf{x}^-_k) = 0$, $\mathrm{(ii)}$ holds because $H_k\E(\mathbf{x}_{k+1} - \mathbf{\hat{x}}_{k+1}^-) = e_k$. Then, there is $g^u_k \sim \mathcal{N}(e_k, \Sigma_k)$ and $\hat{g}^u_k \sim \mathcal{N}(0, \hat{\Sigma}_k)$, and the expected likelihood functional can be evaluated as
   \begin{equation*}
      \begin{aligned}
                                     & \mathcal{L}(\mathscr{F}, \mathbf{y}_{1:n}) \overset{\mathrm{(i)}}{=}  -\mathrm{H}(\mathbf{y}_{1:n})-\sum_{k=1}^{n} D_{KL}\left(g_k^u(\cdot)\Vert \hat{g}^u_k(\cdot)\right) \\
         \overset{\mathrm{(ii)}}{=}  & -\sum_{k=1}^{n}  \mathrm{H}(\mathbf{y}_k \mid \mathbf{x}_k, \mathbf{y}_{1:k-1}) + \mathrm{H}(\mathbf{x}_k \mid \mathbf{x}_{k-1}, \mathbf{y}_{k-1})                         \\
                                     & \qquad + \mathrm{H}(\mathbf{x}_{k-1} \mid \mathbf{y}_{1:k-1}) + D_{KL}\left(g_k^u(\cdot)\Vert \hat{g}^u_k(\cdot)\right)                                                    \\
         \overset{\mathrm{(iii)}}{=} & -\sum_{k=1}^{n}  \mathrm{H}(\mathbf{y}_k \mid \mathbf{x}_k) + \mathrm{H}(\mathbf{x}_k \mid \mathbf{x}_{k-1})                                                               \\
                                     & \qquad + \mathrm{H}(\mathbf{x}_{k-1} \mid \mathbf{y}_{1:k-1}) + D_{KL}\left(g_k^u(\cdot)\Vert \hat{g}^u_k(\cdot)\right),
      \end{aligned}
   \end{equation*}
   where $\mathrm{(i)}$ follows from Theorem \ref{thm:evl-1}, $\mathrm{(ii)}$ follows from the chain rule of conditional entropy, and $\mathrm{(iii)}$ holds because of the Markov property of the stochastic dynamical system. Then, we can calculate each part separately:
   \begin{equation*}
      \begin{aligned}
          & \mathrm{H}(\mathbf{y}_k \mid \mathbf{x}_k) =                     \frac{d_y}{2}(\ln(2\pi) + 1)+\frac{1}{2}\ln|R_k|,                                                                                                            \\
          & \mathrm{H}(\mathbf{x}_k \mid \mathbf{x}_{k-1}) =                 \frac{d_x}{2}(\ln(2\pi) + 1)+\frac{1}{2}\ln|Q_{k-1}|,                                                                                                        \\
          & \mathrm{H}(\mathbf{x}_{k-1} \mid \mathbf{y}_{1:k-1})=            \frac{d_x}{2}(\ln(2\pi) + 1)+\frac{1}{2}\ln|P_{k-1}^+|,                                                                                                      \\
          & D_{KL}\!\left(g_k^u(\cdot)\Vert \hat{g}^u_k(\cdot)\right) \!=\!  \frac{1}{2}\!\!\left[e_k^\top\hat{\Sigma}_k^{-1}e_k \!+\! \tr(\hat{\Sigma}_k^{-1}\Sigma_k) \!-\! \ln\frac{|\Sigma_k|}{|\hat{\Sigma}_k|} \!-\!d_y\!  \right].
      \end{aligned}
   \end{equation*}
   Substituting above expressions into $\mathcal{L}(\mathscr{F}, \mathbf{y}_{1:n})$, it equals
   \begin{equation}
      \begin{aligned}
         -\frac{1}{2}\sum_{k=1}^{n} & \left\{e_k^\top\hat{\Sigma}_k^{-1}e_k\!+\!\tr[\hat{\Sigma}_k^{-1}\Sigma_k]\!+\!\ln |\hat{\Sigma}_k||R_k||Q_{k-1}||P_{k-1}^+|\right. \\
                                    & -\ln|\Sigma_k| +\left.(2\ln(2\pi)+2)d_x + \ln(2\pi)d_y\right\},
      \end{aligned}
   \end{equation}
   and the proof is completed.
\end{appendices}

\bibliographystyle{IEEEtran}
\bibliography{ref}
\begin{IEEEbiographynophoto}{Tao Xu}
   (S'22) received the B.S. degree in the School of Mathematical Sciences from Shanghai Jiao Tong University (SJTU), Shanghai, China. He is currently working toward the Ph.D. degree with the Department of Automation, SJTU. His research interests mainly include differential game, prediction and learning, optimization and control in stochastic dynamical system.
\end{IEEEbiographynophoto}

\begin{IEEEbiographynophoto}{Jianping He}
   (SM'19) is currently an associate professor in the Department of Automation at Shanghai Jiao Tong University. He received the Ph.D. degree in control science and engineering from Zhejiang University, Hangzhou, China, in 2013, and had been a research fellow in the Department of Electrical and Computer Engineering at University of Victoria, Canada, from Dec. 2013 to Mar. 2017. His research interests mainly include the distributed learning, control and optimization, security and privacy in network systems.

   Dr. He serves as an Associate Editor for IEEE Tran. Control of Network Systems, IEEE Open Journal of Vehicular Technology, and KSII Trans. Internet and Information Systems. He was also a Guest Editor of IEEE TAC, International Journal of Robust and Nonlinear Control, etc. He was the winner of Outstanding Thesis Award, Chinese Association of Automation, 2015. He received the best paper award from IEEE WCSP'17, the best conference paper award from IEEE PESGM'17, and was a finalist for the best student paper award from IEEE ICCA'17, and the finalist best conference paper award from IEEE VTC20-FALL.
\end{IEEEbiographynophoto}
\end{document}